\documentclass[a4paper, 10pt]{article}
\usepackage{srcltx,graphicx}
\usepackage{amsmath, amssymb, amsthm}
\usepackage{hyperref}
\usepackage{color, xcolor}
\usepackage{array}
\usepackage{multirow}
\usepackage{psfrag}
\usepackage[hang]{subfigure}
\usepackage{esint}
\usepackage{bigints}

\graphicspath{{images/}}

\newtheorem{theorem}{Theorem}

\newtheorem{lemma}{Lemma}

{\theoremstyle{remark} \newtheorem{remark}{Remark}}
\newtheorem{corollary}{Corollary}

\newcommand\bbR{\mathbb{R}}
\newcommand\bbC{\mathbb{C}}

\newcommand\bbZ{\mathbb{Z}}

\newcommand\bbS{\mathbb{S}}
\newcommand\bbM{\mathbb{M}}

\newcommand\bdf{\boldsymbol{f}}
\newcommand\bc{\boldsymbol{c}}

\newcommand\bx{\boldsymbol{x}}
\newcommand\by{\boldsymbol{y}}
\newcommand\bu{\boldsymbol{u}}
\newcommand\bv{\boldsymbol{v}}
\newcommand\bw{\boldsymbol{w}}
\newcommand\bn{\boldsymbol{n}}

\newcommand\bA{\boldsymbol{\mathrm{A}}}

\newcommand\bG{\boldsymbol{\mathrm{G}}}
\newcommand\bI{\boldsymbol{\mathrm{I}}}
\newcommand\bK{\boldsymbol{\mathrm{K}}}

\newcommand\cA{\mathcal{A}}
\newcommand\cE{\mathcal{E}}
\newcommand\cF{\mathcal{F}}
\newcommand\cI{\mathcal{I}}
\newcommand\cK{\mathcal{K}}
\newcommand\cJ{\mathcal{J}}

\newcommand\cT{\mathcal{T}}

\newcommand\imag{\boldsymbol{\mathrm i}}

\newcommand\dd{\,\mathrm{d}} 
\newcommand\norm[1]{\left\lVert#1\right\rVert}
\newcommand\normLpOmega[2]{\left\lVert#2\right\rVert_{L^{#1}(\Omega)}}
\newcommand\normOperatorLp[2]{\left\lVert#2\right\rVert_{L^{#1}(\Omega)\rightarrow L^{#1}(\Omega)}}
\newcommand\normLpOmegaAni[2]{\left\lVert#2\right\rVert_{L^{#1}(\Omega,\bbM)}}
\newcommand\normOperatorLpAni[2]{\left\lVert#2\right\rVert_{L^{#1}(\Omega,\bbM)\rightarrow L^{#1}(\Omega,\bbM)}}

\numberwithin{equation}{section}

\title{Fast algorithms for integral formulations of steady-state radiative transfer equation}
\author{
  Yuwei Fan\thanks{Department of Mathematics, Stanford University,
    Stanford, CA 94305, email: {\tt ywfan@stanford.edu}},~~
  Jing An\thanks{Institute for Computational and Mathematical
    Engineering, Stanford University, Stanford, CA 94305, email: {\tt
      jingan@stanford.edu}},~~
  Lexing Ying\thanks{Department of Mathematics and Institute for
    Computational and Mathematical Engineering, Stanford University,
    Stanford, CA 94305, email: {\tt lexing@stanford.edu}}
}

\begin{document}
\maketitle

\begin{abstract}
  We investigate integral formulations and fast algorithms for the steady-state radiative transfer
  equation with isotropic and anisotropic scattering. When the scattering term is a smooth
  convolution on the unit sphere, a model reduction step in the angular domain using the Fourier
  transformation in 2D and the spherical harmonic transformation in 3D significantly reduces the
  number of degrees of freedoms. The resulting Fourier coefficients or spherical harmonic
  coefficients satisfy a Fredholm integral equation of the second kind. We study the uniqueness of
  the equation and proved an a priori estimate.  For a homogeneous medium, the integral equation can be
  solved efficiently using the FFT and iterative methods. For an inhomogeneous medium, the recursive
  skeletonization factorization method is applied instead. Numerical simulations demonstrate the
  efficiency of the proposed algorithms in both homogeneous and inhomogeneous cases and for both transport
  and diffusion regimes.
  
  \vspace*{4mm}
\end{abstract}

\section{Introduction}

The radiative transfer equation (RTE) is the primary equation for describing particle propagation
in many different fields, such as neutron transport in reactor physics \cite{pomraning1973equations,
  duderstadt1979transport}, light transport in atmospheric radiative transfer \cite{marshak20053d},
heat transfer \cite{koch2004evaluation} and optical imaging
\cite{klose2002optical,tarvainen2005hybrid}.  In this paper, we focus on the steady state radiative
transfer equation with possibly anisotropic scattering
\begin{equation}\label{eq:RTE}
  \begin{aligned}
    \bv\cdot\nabla_{\bx}\Phi(\bx,\bv)+\mu_t(\bx)\Phi(\bx,\bv) 
    &=\mu_s(\bx)\fint_{\bbS^{d-1}}\hspace{-10pt}\sigma(\bx,\bv\cdot\bv')\Phi(\bx,\bv')\dd\bv'+f(\bx),\quad
    \text{ in } \Omega\times \bbS^{d-1},\\
    \Phi(\bx,\bv) &= 0,\quad \text{ on } \Gamma_{-},
  \end{aligned}
\end{equation}
where the quantity $\Phi(\bx,\bv)$ denotes the photon flux that depends on both space $\bx$ and
angle $\bv$, and $f(\bx)$ is the light source.  $\Omega\in\bbR^d$ is a bounded Lipschitz domain,
$\bbS^{d-1}$ is the unit sphere in $\bbR^d$, $\fint$ is the average integral, and
$\Gamma_{-}=\{(\bx,\bv)\in\partial\Omega\times \bbS^{d-1}: \bn(\bx)\cdot\bv<0\}$ with $\bn(\bx)$
being the outer unit normal vector at $\bx$. The scattering kernel $\sigma(\bx,\bv\cdot\bv')\geq0$
satisfies $\fint_{\bbS^{d-1}}\sigma(\bx,\bv\cdot\bv')\dd\bv'=1$.  The transport coefficient
$\mu_t=\mu_a+\mu_s$ measures the total absorption at $\bx$ due on both the physical absorption
quantified by the term $\mu_a$ and the scattering phenomenon quantified by the term $\mu_s$. To
simplify the discussion, we only consider the vacuum boundary condition and assume that the source
term $f$ is independent of the angular variable $\bv$. General boundary condition and source terms
will be discussed in Section \ref{sec:conclusion}.

Numerical methods for solving the RTE can be categorized into two groups: (a) the probabilistic
approaches, for example Monte Carlo methods \cite{bhan2007condensed, densmore2012hybrid,
  hayakawa2007coupled} and (b) the deterministic schemes based on differential-integral equations.
In the past decades, many deterministic methods have been proposed, including different
discretization in the spatial and angular domains
\cite{adams1998asymptotic,asadzadeh1998finite,manteuffel1999boundary,dedner2002adaptive,turpault2004multigroup,frank2007time,gao2009fast,grella2011sparse,alldredge2016approximating,ren2016fast}
and various preconditioned iterative schemes
\cite{oliveira1998preconditioned,patton2002application,godoy2012parallel}.

However, numerical solutions of RTE still face several challenging issues. One of them is the high
dimensionality of the photon flux $\Phi(\bx,\bv)$, which depends on both the spatial variable $\bx$
and the angular variable $\bv$. For a $d$-dimensional problem, the photon flux therefore depends on
a total of $d+(d-1)=2d-1$ variables. Another issue is the computation of the scattering term
\begin{equation}
  S(\bx,\bv) = \mu_s(\bx) \fint_{\bbS^{d-1}}\sigma(\bx,\bv\cdot\bv')\Phi(\bx,\bv') \dd\bv'.
\end{equation}
Due to the $(d-1)$-dimensional integral in the above formula, a naive computation of the RTE is a
$d+(d-1)+(d-1)=(3d-2)$-dimensional problem. This poses a significant bottleneck for efficient
numerical simulations of the RTE.

Recently in \cite{ren2016fast}, the authors studied the isotropic scattering case
(i.e. $\sigma(\bx,\bv\cdot\bv')\equiv 1$) and converted the RTE into a Fredholm integral equation of the
second kind of the {\em mean local density} $\fint_{\bbS^{d-1}}\Phi(\bx,\bv)\dd\bv$. One can then
directly solve the mean local density first without the solution of the photon flux $\Phi$. When
$\Phi(\bx,\bv)$ is needed, then it can be computed easily by solving a transport equation with a
known right hand side.

Inspired by the approach of \cite{ren2016fast}, in this work we propose integral formulations and
fast algorithms for the RTE with smooth isotropic or anisotropic scattering. The primary idea of the
algorithm is that the angular dependence of the scattering term $S(\bx,\bv)$ can be represented
efficiently using a small number of modes. One can convert the RTE into a Fredholm equation of the
second kind of these modes, thus achieving a significant model reduction. As in \cite{ren2016fast},
once the scattering term $S(\bx,\bv)$ is resolved, applying existing methods (e.g.
\cite{gao2009fast,grella2011sparse}) to the following transport equation with the right hand side
known gives the photon flux $\Phi(\bx,\bv)$
\begin{equation}
  \begin{aligned}
    \bv\cdot\nabla_{\bx}\Phi(\bx,\bv)+\mu_t(\bx)\Phi(\bx,\bv) &=S(\bx,\bv)+f(\bx),\quad
    \text{ in }\Omega\times \bbS^{d-1},\\
    \Phi(\bx,\bv) &= 0,\quad \text{ on } \Gamma_{-}.
  \end{aligned}
\end{equation}
In the rest of this paper, we focus on how to formulate an integral equation for the scattering term
$S(\bx,\bv)$ and how to solve it numerically.

For the isotropic scattering case, a Fredholm integral equation of the second kind of the mean local
density can be directly obtained from the RTE \cite{LionsVol6,ren2016fast}. We study the solvability
of the integral equation based on the contraction principle and prove an a priori estimate of the
solution. The numerical approach for solving the integral equation depends on whether the problem is
spatially constant or not (i.e. whether the coefficients $\mu_t$ and $\mu_s$ independent on
$\bx$). For a homogeneous medium, the Fredholm kernel is simply a convolution, so the FFT can
be used to evaluate the integral. Therefore, the Fredholm equation can be solved by combining
the FFT with a standard iterative method such as MINRES \cite{paige1975minres} and GMRES
\cite{saad1986gmres}.  For an inhomogeneous medium, we use the recursive skeletonization factorization
(RSF) \cite{ho2016hif} method to factorize the dense linear system obtained from discretizing the
integral equation. Due to the special factorization forms used in the RSF, this method is both fast
and accurate.

For the anisotropic scattering case, noticing that the smooth anisotropic scattering term is a convolution
on the unit sphere, we apply the Fourier transform ($d=2$) and the spherical harmonic transform
($d=3$) with respect to $\bv$ to both $\sigma$ and $\Phi$. After truncating the expansion of
$\sigma$ due to its smoothness, we obtain a Fredholm integral equation of the second kind of the
truncated expansion coefficients of $\Phi$. Once solving this integral equation, one obtains the
expansion coefficients of $\Phi$ in the angular variable $\bv$ and thus the scattering term. Similar
to the isotropic case, we study the solvability of the Fredholm equation by the contraction
principle and proved an a priori estimate. Both the FFT-based method and the RSF-based method are
discussed for the homogeneous and inhomogeneous cases, respectively.
Numerical simulations are performed to demonstrate the efficiency of both the FFT-based method and
the RSF-based method.


The rest of the paper is organized as follows. Section \ref{sec:alg1} addresses the isotropic
scattering case, where we review the Fredholm integral equation of the mean local density, prove the
uniqueness of the Fredholm equation, construct the FFT-based method and the RSF-based method, and
report numerical studies. The anisotropic scattering case is studied in Section
\ref{sec:alg2}. 


\section{Isotropic scattering}\label{sec:alg1}
This section studies the isotropic scattering case, i.e.  $\sigma(\bx,\bv\cdot\bv')\equiv 1$. We show the
properties of the RTE, derive the Fredholm equation of the scattering term, construct the fast
algorithms and report numerical studies.


\subsection{Integral formulation}\label{sec:iso_integral}
The integral formulation of the RTE has been studied in the literatures for example in
\cite{case1963existence,LionsVol6,egger2014lp,ren2016fast}. Here for completeness, we provide a
brief review of the integral equation of the RTE in the following, and refer readers to
\cite{LionsVol6} for more details of the derivation.

Define the operators $\cT$ and $\cA$ as
\begin{equation}\label{eq:iso_def_TA}
  \cT \Psi := \bv\cdot\nabla_{\bx}\Psi+\mu_t\Psi,\quad \cA \Psi (\bx)
  :=\fint_{\bbS^{d-1}}\Psi(\bx,\bv)\dd\bv,
\end{equation}
then the RTE \eqref{eq:RTE} can be reformulated as
\begin{equation}\label{eq:RTE_iso_operator}
  \cT\Phi-\mu_s\cA\Phi=f.
\end{equation}
By using the method of characteristics, we obtain
\begin{equation}\label{eq:def_cTinv}
  (\cT^{-1}\Psi) (\bx,\bv) = \int_0^{\tau(\bx,\bv)} 
  \exp\left(-\int_0^s \mu_t\left(\bx - s'\bv\right) ds'\right)
  \Psi(\bx-s\bv,\bv) \dd s,
\end{equation}
where $\tau(\bx,\bv)$ is the distance of a particle traveling from $\bx$ to the domain boundary with
the velocity $-\bv$
\begin{equation}\label{eq:def_tau}
    \tau(\bx,\bv) = \sup\{t: \bx - s\bv \in \Omega \text{ for } 0\leq s< t \}.
\end{equation}
Define the operator $\cJ: \Omega\rightarrow \Omega\times\bbS^{d-1}$ as $\cJ g(\bx,\bv) = g(\bx)$ and
introduce the mean local density as
\begin{equation}\label{eq:iso_density}
  u(\bx) = \cA\Phi(\bx) = \fint_{\bbS^{d-1}}\Phi(\bx,\bv')\dd\bv'.
\end{equation}
The equation \eqref{eq:RTE_iso_operator} can be reformulated as a Fredholm integral equation of the
second kind with the form
\begin{equation}\label{eq:iso_Fredholm}
    u(\bx) = \cK u(\bx) + \tilde{\cK}f(\bx),
\end{equation}
where
\begin{equation}\label{eq:iso_def_K}
    \cK = \cA\cT^{-1}\cJ\mu_s,\quad \tilde{\cK}=\cA\cT^{-1}\cJ.
\end{equation}
Notice that
\begin{equation}\label{eq:iso_derivation}
  \begin{aligned}
    (\tilde{\cK}\phi)(\bx) &= \fint_{\bbS^{d-1}}
    \int_0^{\tau(\bx,\bv)} 
    \exp\left(-\int_0^s \mu_t\left(\bx - s'\bv\right) ds'\right)
    \phi(\bx-s\bv) \dd s\dd\bv\\
    &=\frac{1}{|\bbS^{d-1}|}\int_{\Omega}
    \frac{1}{|\bx-\by|^{d-1}}
    \exp\left(-|\bx-\by|\int_0^1 \mu_t\left(\bx -
    s'(\bx-\by)\right) ds'\right) \phi(\by)\dd\by.
  \end{aligned}
\end{equation}
We have 
\begin{equation}\label{eq:iso_operator2}
  (\tilde{\cK} \phi)(\bx) = \int_{\Omega}\tilde{K}(\bx,\by)\phi(\by)\dd\by,
  \quad
  \tilde{K}(\bx,\by) = \frac{1}{|\bbS^{d-1}|} \frac{E(\bx,\by)}{|\bx-\by|^{d-1}},
\end{equation}
and
\begin{equation}\label{eq:iso_operator}
  (\cK \phi)(\bx) = \int_{\Omega}K(\bx,\by)\phi(\by)\dd\by,
  \quad 
  K(\bx,\by) = \frac{\mu_s(\by)}{|\bbS^{d-1}|}
  \frac{E(\bx,\by)}{|\bx-\by|^{d-1}},
\end{equation}
with $E(\bx,\by) = \exp\left(-|\bx-\by|\int_0^1 \mu_t\left(\bx - s'(\bx-\by)\right) ds'\right)$.
Introducing $\tilde{u}(\bx)=\mu_s(\bx)u(\bx)$ allows one to rewrite the integral equation
\eqref{eq:iso_Fredholm} as
\begin{equation}\label{eq:iso_Fredholm2}
  \tilde{u}(\bx)=\mu_s(\bx) (\tilde{\cK}\tilde{u})(\bx)
  +\mu_s(\bx)(\tilde{\cK}f)(\bx),
\end{equation}
or
\begin{equation}\label{eq:iso_Fredholm2_2}
  \frac{\tilde{u}(\bx)}{\mu_s(\bx)}=(\tilde{\cK}\tilde{u})(\bx)
  +(\tilde{\cK}f)(\bx) \tag{\ref{eq:iso_Fredholm2}'}.
\end{equation}
The advantage of the formulation \eqref{eq:iso_Fredholm2_2} is the symmetry of the kernel
$\tilde{K}(\bx,\by)$, i.e $\tilde{K}(\bx,\by)=\tilde{K}(\by,\bx)$, which can be used to save memory
and calculation in numerical method.


\subsection{Existence theory}

The solvability of the RTE \eqref{eq:RTE} was studied in the past decades \cite{case1963existence,
  LionsVol6, choulli1999,stefanov2008inverse, schlottbom2011forward, egger2014lp}. Dautray and Lions
in \cite{LionsVol6} developed a rather complete $L^p$ theory for the RTE in the framework of
semi-group theory. However, its a priori estimates for the solution are based on a rather strong
assumption on $\mu_s$ and $\mu_t$. In \cite{stefanov2008inverse}, Stefanov and Uhlmann studied the
existence of solutions of RTE based on the compactness argument and the Fredholm theory. Based on
the contraction principle, Egger and Schlottbom presented a $L^p$ theory for RTE with relaxed
conditions on the parameters $\mu_t$, $\mu_s$, and $\sigma$ and also gave a sharp a priori estimate
of the solution in \cite{schlottbom2011forward, egger2014lp, egger2014stationary}. The aim of this
subsection is to translate the work of the RTE to the Fredholm equation case
\eqref{eq:iso_Fredholm}, study the properties of the operators $\cK$ and $\tilde{\cK}$ along with
the solvability of the Fredholm equations \eqref{eq:iso_Fredholm}, and present an a priori estimate
of the solution. We start from the following assumptions:
\begin{enumerate}
\item[(C1)] $\Omega\in\bbR^d$ is a bounded Lipschitz domain;
\item[(C2)] $\mu_t, \mu_s:\Omega\rightarrow\bbR$ are non-negative and bounded on $\Omega$, and
  $\mu_s(\bx)\leq \mu_t(\bx)$.
\end{enumerate}
The main analytic result is the following theorem.

\begin{theorem}\label{thm:existence}
    Under assumptions (C1) and (C2) hold, for any $1\leq p\leq\infty$ and $f\in L_p(\Omega)$,
    the Fredholm equation \eqref{eq:iso_Fredholm} admits a unique solution $u\in L^p(\Omega)$ which
    satisfies
    \begin{equation}\label{eq:iso_unique}
      \normLpOmega{p}{u}
      \leq \tau \exp\left( \tau\normLpOmega{\infty}{\mu_s} \right)
      \normLpOmega{p}{f},
    \end{equation}
    where $\tau=\sup_{\bx\in\Omega,\bv\in\bbS^{d-1}}\tau(\bx,\bv)$.
\end{theorem}
The proof is divided into two cases: $p=\infty$ and $1\leq p<\infty$. For the case $p=\infty$, the
following result holds.
\begin{lemma}\label{lem:Linfcontraction}
  Under assumptions (C1) and (C2), the linear operator $\cK: L^\infty(\Omega)\rightarrow
  L^\infty(\Omega)$ is a contraction with
  \begin{equation}\label{eq:cKLinf}
    \normOperatorLp{\infty}{\cK} \leq C := 1 - \exp\left( -\tau \normLpOmega{\infty}{\mu_s} \right).
  \end{equation}
    Assume $f\in L^\infty(\Omega)$, then
    \begin{equation}\label{eq:KfLinf}
      \normLpOmega{\infty}{\tilde{\cK}f}\leq \tau\normLpOmega{\infty}{f}.
    \end{equation}
\end{lemma}
\begin{proof}
  Let $E^{[\mu_t]}(\bx,\bv,s)=\exp\left( - \int_0^s\mu_t(\bx-s'\bv)\dd s'\right)$, then
  \begin{equation}\label{eq:relation_Emu}
    E^{[\mu_t]}(\bx+s\bv,\bv,s)=E^{[\mu_t]}(\bx,\bv,s).
  \end{equation}
  Recalling the definition of $\cK$ \eqref{eq:iso_def_K}, for $\phi(\bx)\in L^\infty(\Omega)$, we
  have
  \begin{align*}
    |(\cK \phi)(\bx) | 
    &= |(\tilde{\cK}(\mu_s\phi))(\bx)|
    \leq \normLpOmega{\infty}{\phi}|(\tilde{\cK}\mu_s)(\bx)| .
  \end{align*}
  Since $\mu_t(\bx)\geq\mu_s(\bx)$ are bounded and non-negative, we obtain
  \begin{equation}\label{eq:mu_inequality}
    \begin{aligned}
      0\leq (\tilde{\cK}\mu_s)(\bx) &=
      \fint_{\bbS^{d-1}}  \int_0^{\tau(\bx,\bv)} E^{[\mu_t]}(\bx,\bv,s)      \mu_s(\bx-s\bv)  \dd s \dd \bv \\
      &\leq \fint_{\bbS^{d-1}}  \int_0^{\tau(\bx,\bv)} E^{[\mu_s]}(\bx,\bv,s)      \mu_s(\bx-s\bv)  \dd s \dd \bv \\
      &\leq \fint_{\bbS^{d-1}} \left(1-\exp\left(-\normLpOmega{\infty}{\mu_s}\tau(\bx,\bv)\right) \right)\dd \bv\leq C < 1.
    \end{aligned}
  \end{equation}
  Therefore $\normLpOmega{\infty}{\cK \phi(\bx)} \leq C\normLpOmega{\infty}{\phi} <
  \normLpOmega{\infty}{\phi}$ for any $\phi\in L^{\infty}(\Omega)$, i.e.  $\cK$ is a contraction in
  $L^\infty(\Omega)$, and $\normOperatorLp{\infty}{\cK}\leq C$.

  Recalling the definition of $\tilde{\cK}$ \eqref{eq:iso_def_K}, and noticing
  $0<E^{[\mu_t]}(\bx,\bv,s)\leq 1$, we have
  \begin{equation*}
    \begin{aligned}
      |\tilde{\cK}f|&\leq\left| \fint_{\bbS^{d-1}}\int_0^{\tau(\bx,\bv)} E^{[\mu_t]}(\bx,\bv,s)
      |f(\bx-s\bv')|\dd s\dd\bv\right|\\
      &\leq \fint_{\bbS^{d-1}}\int_0^{\tau(\bx,\bv)}|f(\bx-s\bv)|\dd s\dd\bv \leq \tau
      \normLpOmega{\infty}{f}.
    \end{aligned}
  \end{equation*}
\end{proof}
For the case $1\leq p<\infty$, notice that \eqref{eq:iso_Fredholm} can be rewritten as $u =
\tilde{\cK}\left( \mu_s u+f \right)$. By letting $w=\mu_s u+f$, the Fredholm equation
\eqref{eq:iso_Fredholm} can be stated equivalently as
\begin{equation}\label{eq:iso_Fredholm_w}
    w=\mu_s\tilde{\cK}w+f, \quad u = \tilde{\cK} w.
\end{equation}
\begin{lemma}\label{lem:Lpcontranction}
    Under assumptions (C1), (C2), the linear operator $\mu_s\tilde{\cK}: L^p(\Omega)\rightarrow
    L^p(\Omega)$, $1\leq p<\infty$ is a contraction with
    \begin{equation}\label{eq:cKLp}
      \normOperatorLp{p}{\mu_s\tilde{\cK}} \leq C,
    \end{equation}
    and 
    \begin{equation}\label{eq:uwLp}
      \normLpOmega{p}{u}\leq \tau\normLpOmega{p}{w}.
    \end{equation}
\end{lemma}
\begin{proof}
    Recalling the definition of $\tilde{\cK}$ \eqref{eq:iso_def_K}, for any $\phi\in L^p(\Omega)$,
    we have
    \begin{equation*}
        \begin{aligned}
            \normLpOmega{p}{\mu_s\tilde{\cK}\phi}^p &\leq
            \int_{\Omega}\left(\mu_s(\bx)
            \fint_{\bbS^{d-1}}  \int_0^{\tau(\bx,\bv)} 
            E^{[\mu_t]}(\bx,\bv,s)
            |\phi(\bx-s\bv)| \dd s \dd \bv\right)^p\dd\bx\\
            &= \int_{\Omega}\left(
            \fint_{\bbS^{d-1}}  \int_0^{\tau(\bx,\bv)} 
            E^{[\mu_t]}(\bx,\bv,s) \mu_s(\bx-s\bv)
            \dd s \dd \bv\right)^p|\phi(\bx)|^p\dd\bx\\
            &\leq C^p\normLpOmega{p}{\phi}^p<\normLpOmega{p}{\phi}^p,
        \end{aligned}
    \end{equation*}
    where the relation \eqref{eq:relation_Emu} and \eqref{eq:mu_inequality} are used in the equality
    and the last inequality, respectively. Therefore, $\mu_s\tilde{\cK}$ is a contraction in
    $L^p(\Omega)$ and $\normOperatorLp{p}{\mu_s\tilde{\cK}}\leq C$.

    The non-negativeness of $\mu_t$ indicates $0\leq E^{[\mu_t]}(\bx,\bv,s)\leq1$, thus we have
    \begin{equation*}
      \normLpOmega{p}{u}^p \leq \int_{\Omega} \left(\fint_{\bbS^{d-1}} \int_0^{\tau(\bx,\bv)}
      E^{[\mu_t]}(\bx,\bv,s)|w(\bx-s\bv)|\dd s\dd\bv\right)^p\dd\bx \leq \normLpOmega{p}{\tau w}^p.
    \end{equation*}
\end{proof}
With this two lemmas, the proof of the Theorem \ref{thm:existence} is rather straightforward.
\begin{proof}[Proof of Theorem \ref{thm:existence}]
  For the case $p=\infty$, noticing Lemma \ref{lem:Linfcontraction}, and applying Banach's fixed
  point theorem \cite[Chapter 3, Theorem 3.2]{hunter2001applied}, we obtain that
  \eqref{eq:iso_Fredholm} has a unique solution $u\in L^\infty(\Omega)$ if $\tilde{\cK}f\in
  L^\infty(\Omega)$, which is guaranteed by \eqref{eq:KfLinf}.
  
  Analogously, for the case $1\leq p<\infty$, noticing Lemma \ref{lem:Lpcontranction} and applying
  Banach's fixed point theorem, we obtain that \eqref{eq:iso_Fredholm_w} has a unique solution $w\in
  L^p(\Omega)$, which indicates \eqref{eq:iso_Fredholm} has a unique solution $u\in L^p(\Omega)$.
  The conclusion \eqref{eq:iso_unique} is a directly deduction of Lemma \ref{lem:Lpcontranction}.
  This completes the proof.
\end{proof}

By setting $p=2$ in Lemma \ref{lem:Lpcontranction}, we can directly deduce the result.
\begin{corollary}\label{cor:spd}
  Under assumptions (C1) and (C2), the linear operators $\cI-\cK$, $\cI-\mu_s\tilde{\cK}:$
  $L^2(\Omega)\rightarrow L^2(\Omega)$ are positive definite.
\end{corollary}
\begin{proof}
  Since $\cK=\tilde{\cK}\mu_s$, and $\tilde{\cK}$ is a symmetric operator, for any $\phi\in
  L^2(\Omega)$, we have
  \begin{equation*}
      \langle\phi, \cK\phi\rangle 
      =\langle \tilde{\cK}\phi,\mu_s\phi\rangle
      =\langle \mu_s \tilde{\cK}\phi,\phi\rangle.
  \end{equation*}
  Using \eqref{eq:cKLp} with $p=2$ and H{\"o}lder inequality, we can
  obtain 
  \begin{equation*}
      \left\langle \left(\cI-\mu_s \tilde{\cK}\right)\phi,\phi\right\rangle \geq
      (1-C)\normLpOmega{2}{\phi}^2.
  \end{equation*}
  Since $C<1$, $\cI-\cK$ and $\cI-\mu_s\tilde{\cK}$ are both positive definite.
\end{proof}

\subsection{Fast algorithms}

In this subsection, we assume that $\mu_s$ has a positive lower bound and study the fast algorithms
for solving \eqref{eq:iso_Fredholm2_2}. These algorithms can also be applied to
\eqref{eq:iso_Fredholm}. Solving \eqref{eq:iso_Fredholm2} requires two steps:
\begin{enumerate}
\item[S.1] Evaluate $(\tilde{\cK}f)(\bx)$; 
\item[S.2] Given $r(\bx)$, solve $\left( \frac{\cI}{\mu_s(\bx)}
  -\tilde{\cK}\right)\tilde{u}(\bx)=r(\bx)$,
  or equivalently $\tilde{u}(\bx) = \left( \frac{\cI}{\mu_s(\bx)}-\tilde{\cK} \right)^{-1}r(\bx)$.
\end{enumerate}

To simplify the discussion, we assume that the domain $\Omega$ is rectangular. The discretization of
the Fredholm equation \eqref{eq:iso_Fredholm2_2} starts by partitioning the domain $\Omega$ into
Cartesian cells of equal volume. Let $N$ be the number of the cells and $\{\bx_i\}_{1\le i \le N}$
be the center of the cells. We introduce $\tilde{u}_i$ as an approximation of $\tilde{u}(\bx_i)$ and
define $\mu_{s,i}=\mu_s(\bx_i)$ and $f_i=f(\bx_i)$. Using the Nystr\"om method, the discretization
of \eqref{eq:iso_Fredholm2_2} takes the following form
\begin{equation}\label{eq:iso_discretization}
  \frac{\tilde{u}_i}{\mu_{s,i}}=\sum_{j=1}^N\tilde{K}_{i,j}\tilde{u}_j+\sum_{j=1}^N\tilde{K}_{i,j}f_j.
\end{equation}
When $i\not=j$, $\tilde{K}_{i,j}$ is equal to $\tilde{K}(\bx_i,\bx_j)$ scaled by the cell
volume. When $i=j$ since $\tilde{K}(\bx,\by)$ is singular at $\bx=\by$, the value $\tilde{K}_{i,i}$
is taken to be the integral of $\tilde{K}$ at the $i$-th cell.

Setting $\tilde{\bK}=\left( \tilde{K}_{ij} \right)_{N\times N}$ and
$\bA=\left(\frac{\delta_{ij}}{\mu_{s,i}}-\tilde{K}_{ij}\right)_{N\times N}$ allows one to rewrite
\eqref{eq:iso_discretization} in the following matrix form
\begin{equation}\label{eq:iso_AU}
  \bA \tilde{\bu} = \tilde{\bK}\bdf,
\end{equation}
where $\tilde{\bu}=\left( \tilde{u}_i \right)$ and $\bdf=\left( f_i \right)$. Corollary
\ref{cor:spd} guarantees the matrix $\bA$ is positive definite. Finally, in order to get the
approximation $\bu=\left(u_i\right)$ for the values of $u(\bx)$ at $\bx_i$, we simply set
$u_i = \tilde{u}_i/\mu_{s,i}$.

In the calculation of $\tilde{K}_{ij}$, there is a line integral in $E(\bx_i,\bx_j)$. The approach
of \cite{ren2016fast} uses the Fourier transform of $\mu_t$ to evaluate the integral. The
computational cost is proportional to the number of the Fourier modes used in the integral. Even
$\mu_t$ is quite smooth, this method can be somewhat inefficient in practice. Here we use the
Gauss-Legendre quadrature for the line integral. In all tests of this paper, the Gauss-Legendre
quadrature with merely $5$ points give sufficiently accurate results for smooth $\mu_t$.

The calculation of $\tilde{K}_{ii}$ requires evaluating an integral of a singular function. Noticing
that $\tilde{K}(\bx,\by) = \frac{1}{|\bbS^{d-1}||\bx-\by|^{d-1}} +
\frac{1}{|\bbS^{d-1}|}\frac{E(\bx,\by)-1}{|\bx-\by|^{d-1}}$, we first evaluate the mean of the first
part by analysis and then apply the Gauss-Legendre quadrature for the second (smooth)
part. Numerical test shows the Gauss-Legendre quadrature with $15$ points gives accurate results.

To solve \eqref{eq:iso_AU}, a typical direct method proceeds by constructing $\bA$, factorizing it
with either LU or Cholesky decomposition, and solving it with backward/forward substitution. The
computational cost of such a direct method scales like $O(N^3)$, where $N$ is the number of discrete
points of $\Omega$. This is rather costly where $N$ is large and in what follows we study fast
alternatives.

\subsubsection{Homogeneous media and FFT-based algorithm}\label{sec:isoFFT}
For a homogeneous medium, the total transport coefficient $\mu_t$ is independent on $\bx$. The kernel
\begin{equation}
  \tilde{K}(\bx,\by) = \frac{1}{|\bbS^{d-1}|}
  \frac{\exp\left(-\mu_t|\bx-\by|\right)}{|\bx-\by|^{d-1}} := \kappa(\bx-\by)
\end{equation}
only depends on $|\bx-\by|$ and this indicates that $\tilde{\cK} f(\bx)$ is a convolution. If the
domain $\Omega$ is a rectangle ($d=2$) or cuboid ($d=3$), the Fourier transformation can be used to
calculate the convolution more efficiently. More precisely, let us assume $\Omega=[0,1]^d$ and then
$\kappa(\bx)$ is defined on $[-1,1]^d$. One then extends $\kappa(\bx)$ periodically to the whole
space and redefines $\tilde{K}(\bx,\by)=\kappa(\bx-\by)$, $\bx,\by\in\bbR^d$.  For any function
$\phi(\bx)$, $\bx\in[0,1]^d$, extending it to $[-1,1]^d$ by padding zero (and still denoting it by
$\phi$), we obtain
\begin{equation}
  \tilde{\cK}\phi(\bx) = \int_{[-1,1]^d}\kappa(\bx-\by)\phi(\by)\dd\by,
\end{equation}
which can be evaluated by the Fourier transformation. If the discretization in
\eqref{eq:iso_discretization} is uniform, then the matrix $\tilde{\bK}=(\tilde{K}_{i,j})_{N\times
  N}$ is a circulant matrix, thus the FFT can be employed to evaluate $\tilde{\bK}\tilde{\bu}$ and
$\tilde{\bK}\bdf$, with the computation cost reduced to $O(N\log(N))$ with $N=n^d$.  Hence, the
application of the operators $\tilde{\cK}$ and $\frac{\cI}{\mu_s(\bx)}-\tilde{\cK}$ is highly
efficient and Step S.2 can be evaluated by a standard iterative method, for example MINRES
\cite{paige1975minres} and GMRES\cite{saad1986gmres}.


Before the evaluation of $\tilde{\cK}\phi(\bx)$, one needs to calculate the Fourier modes of
$\kappa(\bx)$, which costs $O(N\log(N))$ computation steps and $O(N)$ storage. Therefore, the
precomputing cost of the algorithm is $O(N\log(N))$. The computation cost of the two steps S.1 and
S.2 are $O(N\log(N))$ and $O(n_{iter}N\log(N))$, respectively, where $n_{iter}$ the number of
iterations used in the iterative method. As a result, the total computation cost is
$O(n_{iter}N\log(N))$.

\begin{remark}
  The algorithm presented in this subsection can be applied to \eqref{eq:iso_Fredholm} without any
  difficulty. The differences are that we also have to assume $\mu_s$ is independent on $\bx$, and
  in solving Step S.2, the iterative method including MINRES does not work because the kernel
  $K(\bx,\by)$ is not symmetric.
\end{remark}

\subsubsection{Inhomogeneous media and RSF-based algorithm}\label{sec:isoRSF}

The FFT-based algorithm in Section \ref{sec:isoFFT} has a strong requirement that the total transport
coefficient $\mu_t$ is constant. Moreover, Step S.2 dominates the computation cost of the whole
algorithm. When the number of iterations is large or there are multiple right hand sides, the
iterative methods can be somewhat inefficient. The fast multipole based algorithm in
\cite{ren2016fast} shares the same issue since an iterative method is also used in Step S.2.

The recently proposed {\em recursive skeletonization factorization} (RSF) in \cite{ho2016hif-ie}
provides an alternative.  Building on top of earlier work in \cite{greengard2009fast, ho2012fast},
the RSF constructs a multiplicative factorization of the matrix operator using a sparsification
technique called {\em interpolative decomposition} recursively. This representation enables
surprisingly simple algorithms to apply $\bA$ and $\bA^{-1}$ directly. This makes it rather easy
to solve \eqref{eq:iso_AU}. RSF is based on elimination, so it requires that certain intermediate
matrices be invertible, which is guaranteed by Corollary \ref{cor:spd}. For more details of RSF, we
refer the readers to \cite[Section 3]{ho2016hif-ie}.

When the matrix $\bA$ is factorized recursively, the RSF takes advantage of the low-rank behavior of
the off-diagonal entries of the matrix and thus avoids visiting all entries of $\bA$. More precisely
from \cite{ho2016hif-ie}, for a Fredholm integral equation of the second kind, the RSF only visits
$O(N\log(N))$ entries of $\bA$. The cost $t_f$ of factorizing $\bA$ and the cost $t_{a/s}$ of
evaluating $\bA\bdf$ as well as $\bA^{-1}\bdf$ are respectively given by
\begin{equation} \label{eq:time}
  t_{f} = 
  \begin{cases}
    O(N^{3/2}), & d=2\\
    O(N^2),    & d=3,
  \end{cases}
  \qquad
  t_{a/s} =
  \begin{cases}
    O(N\log N),  & d=2,\\
    O(N^{4/3}),  & d=3.\\
  \end{cases}
\end{equation}
Thus, the precomputation cost of the algorithm is $t_f$, and the costs of Steps S.1 and S.2 are
both $t_{a/s}$, with no iterations involved.

As we pointed out, the FFT-based algorithm has strong requirements on the total transport coefficient
$\mu_t$ (i.e., constant) and the domain $\Omega$, while the RSF-based works for any $\mu_t$ and
domain $\Omega$. However, even for the case of constant $\mu_t$, sometimes it is preferred to use
the RSF-based algorithm.  The nice feature of the RSF-based algorithm is that, for any additional
right hand side, the extra cost is merely $t_{a/s}$. Therefore, when one needs to solve
\eqref{eq:iso_Fredholm2} for many different source terms $f$, the RSF-based algorithm shows its
advantage.

\begin{remark}
  Since the kernel $\tilde{K}(\bx,\by)$ is symmetric, the matrix $\tilde{\bA}$ is symmetric as well.
  RSF will use this property to save storage memory and factorization time (save nearly half memory
  and time \cite{ho2016hif}).  The RSF-based algorithm can be directly applied on
  \eqref{eq:iso_Fredholm}. However, due to the lack of symmetry in the kernel $K(\bx,\by)$, RSF has
  to factorize a non-symmetric matrix with extra computational and storage costs.
\end{remark}

\subsection{Numerical results}

\newcommand{\pre}{\text{pre}}
\newcommand{\sol}{\text{sol}}

\newcommand{\DIR}{\text{DIR}}
\newcommand{\FFT}{\text{FFT}}
\newcommand{\RSF}{\text{RSF}}

Here we provide several numerical examples to study the complexity and accuracy of the direct
method, the FFT-based method and the RSF-based method. The computational cost consists of two parts:
the precomputation part and the solution part, denoted by $T_{\pre}$ and $T_{\sol}$,
respectively. The precomputation part is the time used for preparing and inverting $\bA$ for the
direct method; preparing the convolution vector for the FFT-based method; or factorizing $\bA$ for
the RSF-based method.

The solution part is the time used for solving a single linear system \eqref{eq:iso_AU}, i.e.,
applying $\bA^{-1}$ for the direct method; applying GMRES/MINRES for the FFT-based method; or
applying the factorized inverse $\bA^{-1}$ for the RSF-based method. Here, we shall use the
superscript $\DIR$, $\FFT$ and $\RSF$ to denote the corresponding quantities for these three
methods, respectively. For example, $T_{\pre}^{\RSF}$ denotes the precomputation time of the
RSF-based method. All the numerical simulations are performed in MATLAB R2016b on a single core
(without parallelization) of an Intel(R) Xeon(R) CPU E5-1620 0 @ 3.60GHz on a 64-bit Linux desktop
with 32GB of RAM.

The parameter $\epsilon$ is used to control the designed accuracy of the solution. In the FFT-based
method $\epsilon$ is the desired tolerance of the GMRES/MINRES method, while in the RSF-based method
$\epsilon$ is the relative accuracy parameter of the RSF. To measure the accuracy, we introduce the
relative $L^2$ error
\begin{equation}
    \cE = \frac{\norm{\bu-\bu_{\text{ref}}}_{L^2}}{\norm{\bu_{\text{ref}}}_{L^2}},
\end{equation}
where $\bu$ is the solution obtained either from the FFT-based or from the RSF-based method and
$\bu_{\text{ref}}$ is the reference solution. Similar to the computational cost, we use the
superscripts $\FFT$ and $\RSF$ to denote the relative error of these two methods, respectively.

\begin{figure}[ht]
  \subfigure[$f_1$ in \eqref{eq:source1}]{
    \includegraphics[width=0.3\textwidth]{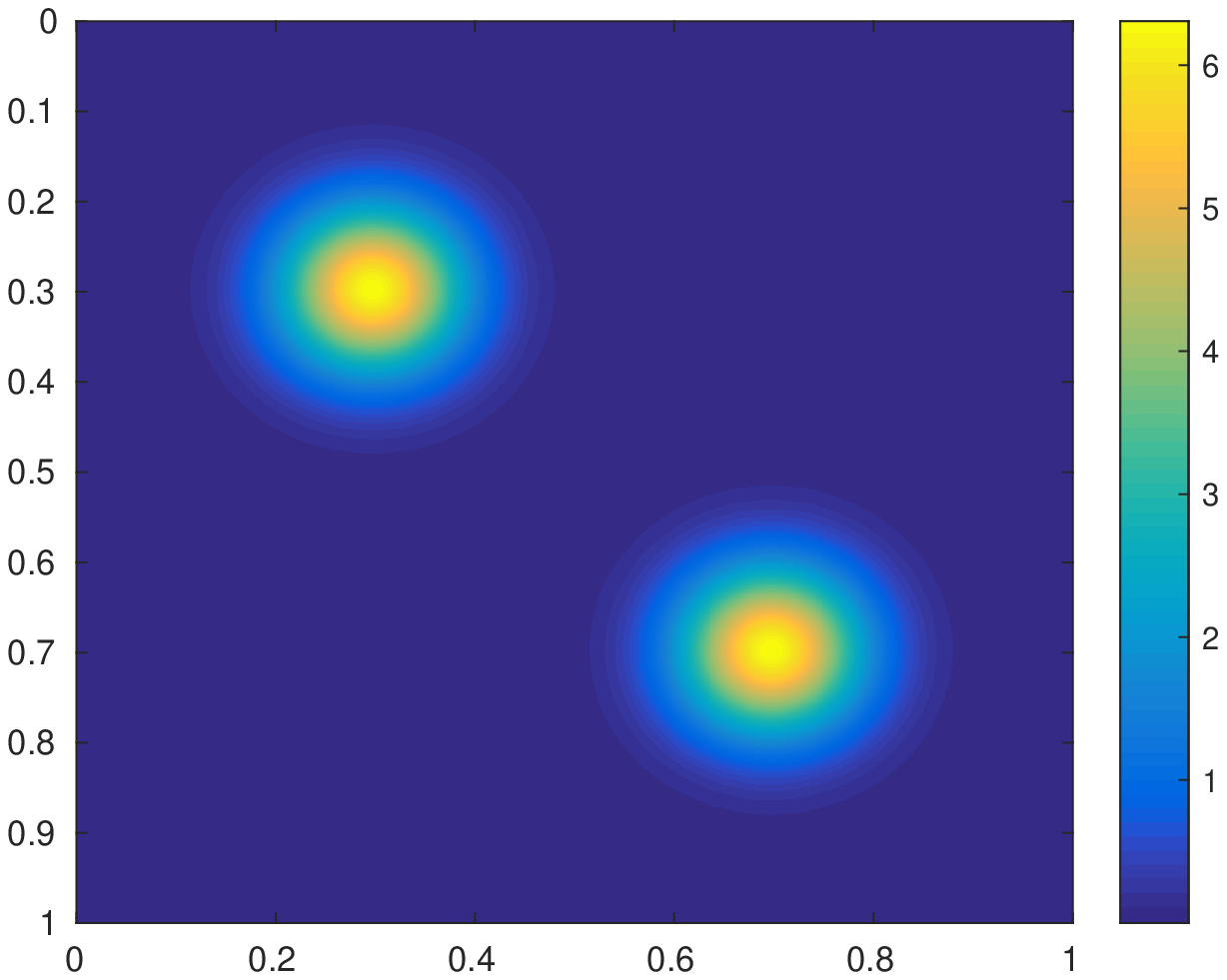}
  }
  \subfigure[$f_2$ in \eqref{eq:source2}]{
    \includegraphics[width=0.3\textwidth]{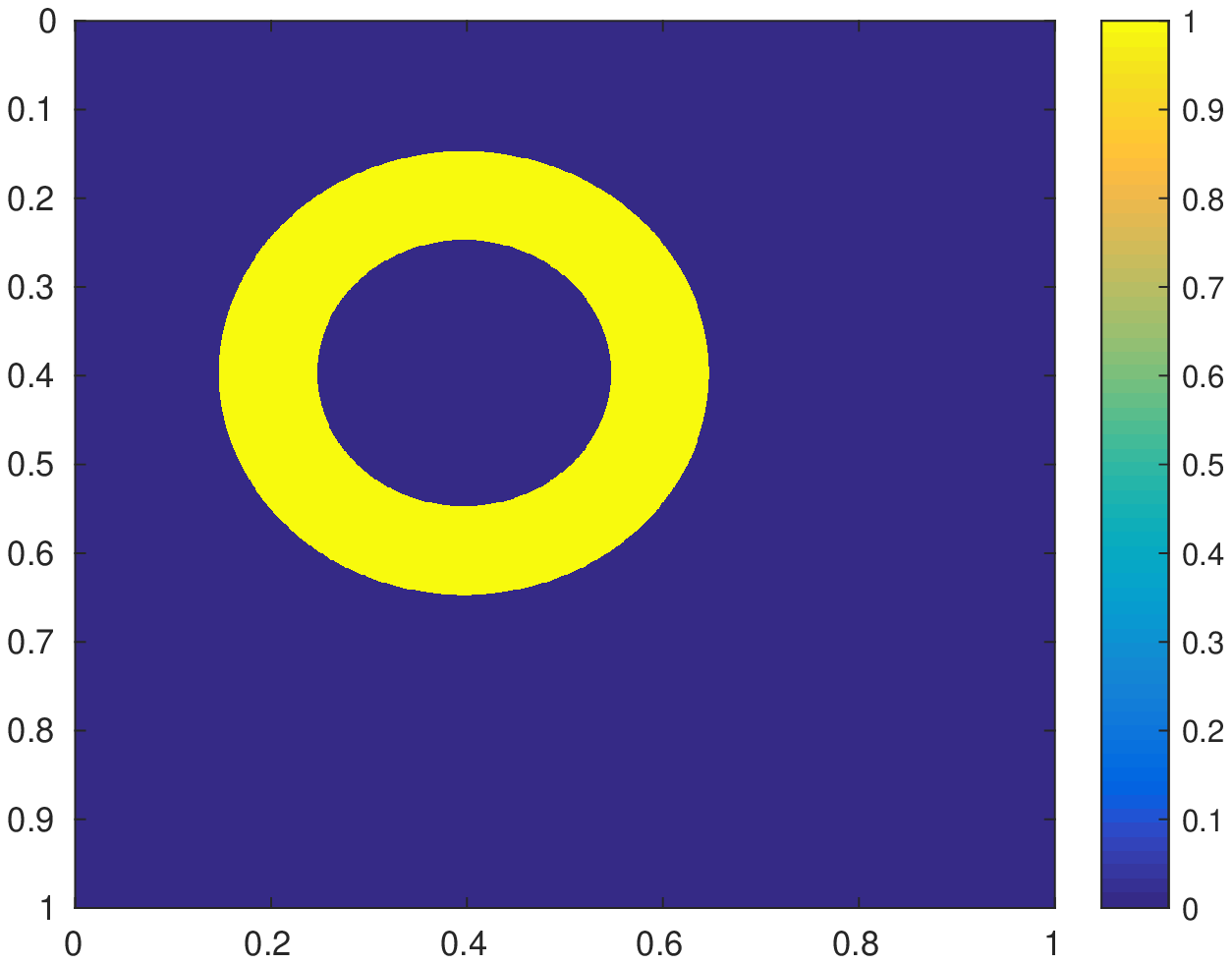}
  }
  \subfigure[$f_3$ in \eqref{eq:source3}]{
    \includegraphics[width=0.3\textwidth]{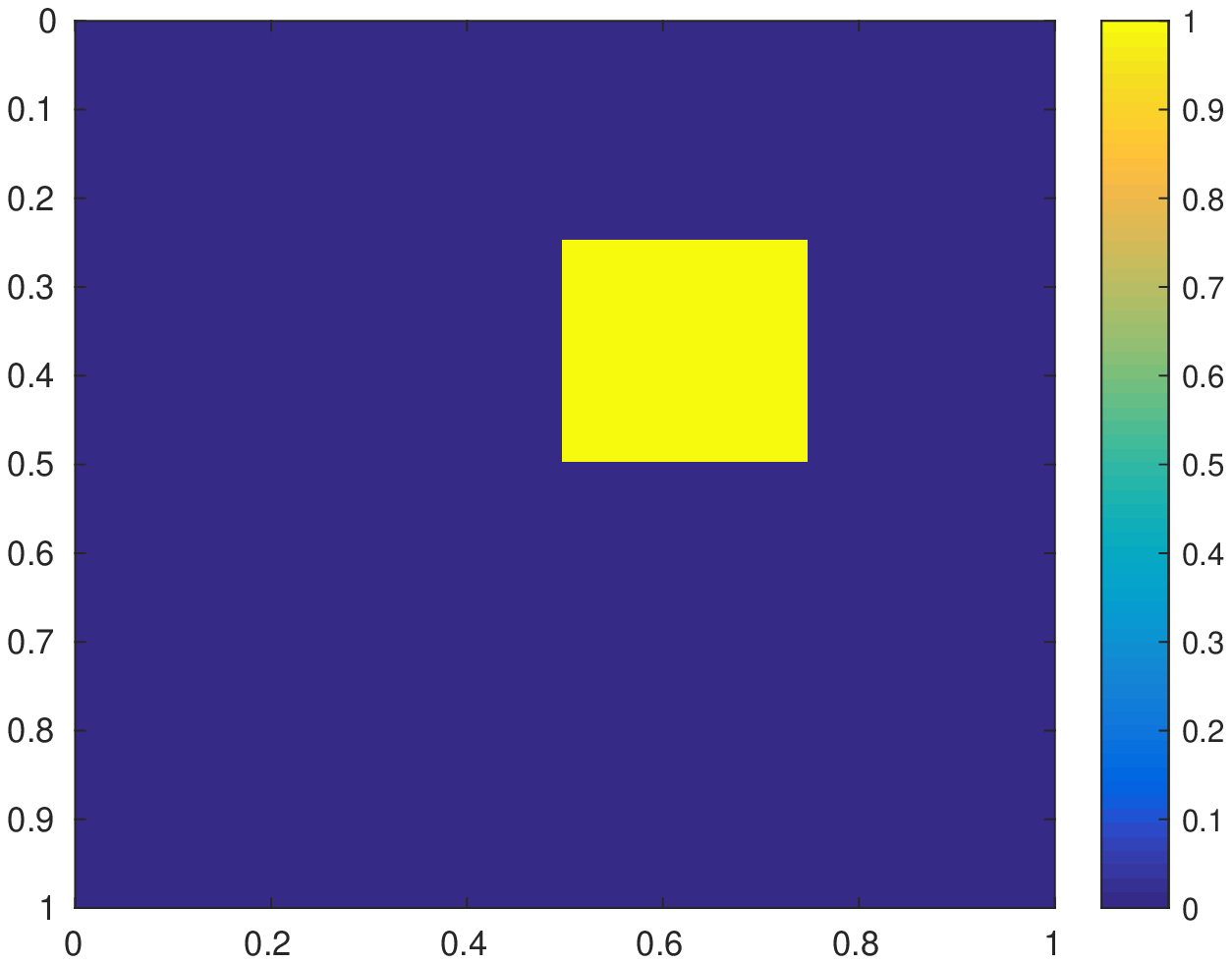}
  }
  \caption{\label{fig:source}Source terms \eqref{eq:source} used in
    numerical simulation.}
\end{figure}

Three source terms used in the simulations are
\begin{subequations} \label{eq:source}
  \begin{align}
    \label{eq:source1}
    f_1(\bx)&=
    \frac{1}{\sqrt{2\pi T}} \exp\left(-\frac{|\bx-\bc_1|^2}{2T}\right)
    +\frac{1}{\sqrt{2\pi T}} \exp\left(-\frac{|\bx-\bc_2|^2}{2T}\right),\\
    \label{eq:source2}
    f_2(\bx)&=\left\{\begin{array}{ll}
        1 & \text{if }|\bx-\bc_3|\in[3/20,1/4],\\
        0 & \text{otherwise},
    \end{array}\right.\\
    \label{eq:source3}
    f_3(\bx) &= \left\{\begin{array}{ll}
        1, & \text{if } x_1,x_2\in[1/2,3/4],\\
        0, & \text{otherwise},
    \end{array}\right.
\end{align}
\end{subequations}
with $T=4E-3, \bc_{1}=(3/10,7/10), \bc_{2}=(7/10,3/10), \bc_{3}=(3/5,2/5)$. Figure \ref{fig:source}
plots the contour of these source terms. The computational domain $\Omega$ is chosen as the unit
square, i.e.  $\Omega=[0,1]^2$. The domain is uniformly discretized by $n$ points in each direction.


These numerical examples test for different $n$, different tolerance $\epsilon$, and different
scattering coefficients $\mu_s$ and $\mu_t$. Because the FFT-based method only works for the
constant transport coefficient case, we report the constant transport coefficient case and the
variable coefficient case separately.

\subsubsection{Homogeneous media}

\paragraph{Computational cost}
We perform simulation for different mesh size $n=32$, $64$, $128$, $256$ and $512$ and different
tolerance $\epsilon$ as $1E-4$, $1E-6$ and $1E-8$, and for the constant scattering coefficient
$\mu_s=1$ and $\mu_t=1.2$ to study the computational cost of the three methods. The results are
summarized in Table \ref{tab:iso_con_time}.  The source term is given in \eqref{eq:source1}.  The
FFT-based method is the fastest in both precomputation and solution. Comparing the solution time,
one finds that the RSF-based method is only a little slower than the FFT-based method. Both the
RSF-based method and direct method need to decompose the matrix, but the RSF is faster in both
precomputation and solution steps.

\begin{table}[ht]
    \centering
    \begin{tabular}{rr|ccc|ccc}
        \hline
        \rule{0pt}{3ex}$n$   & $\epsilon$            &
        $T_{\pre}^{\DIR}$(sec) & $T_{\pre}^{\FFT}$(sec)  &
        $T_{\pre}^{\RSF}$(sec)  & $T_{\sol}^{\DIR}$(sec)  &
        $T_{\sol}^{\FFT}$(sec) & $T_{\sol}^{\RSF}$(sec)  \\ \hline
        32  & $1E-4$  & $3.39E-2$  &   $3.23E-4$  &   $5.75E-2$ & $1.86E-3$  &   $2.99E-3$  &   $2.53E-3$\\ 
        64  & $1E-4$  & $9.19E-1$  &   $9.45E-4$  &   $4.34E-1$ & $2.44E-2$  &   $9.32E-3$  &   $1.27E-2$\\ 
        128 & $1E-4$  & $5.04E+1$  &   $1.76E-3$  &   $3.32E+0$ & $3.78E-1$  &   $1.52E-2$  &   $6.91E-2$\\ 
        256 & $1E-4$  &      --    &   $8.65E-3$  &   $2.36E+1$ &      --    &   $9.71E-2$  &   $3.51E-1$\\ 
        512 & $1E-4$  &      --    &   $3.97E-2$  &   $1.93E+2$ &      --    &   $2.83E-1$  &   $1.78E+0$\\ 
        \hline
        32  & $1E-6$  & $3.39E-2$  &   $3.23E-4$  &   $8.08E-2$ & $1.86E-3$  &   $3.16E-3$  &   $2.67E-3$\\ 
        64  & $1E-6$  & $9.19E-1$  &   $9.45E-4$  &   $7.24E-1$ & $2.44E-2$  &   $1.02E-2$  &   $1.38E-2$\\ 
        128 & $1E-6$  & $5.04E+1$  &   $1.76E-3$  &   $6.12E+0$ & $3.78E-1$  &   $1.74E-2$  &   $8.05E-2$\\ 
        256 & $1E-6$  &      --    &   $8.65E-3$  &   $4.43E+1$ &      --    &   $9.52E-2$  &   $4.41E-1$\\ 
        512 & $1E-6$  &      --    &   $3.97E-2$  &   $3.87E+2$ &      --    &   $3.33E-1$  &   $2.21E+0$\\ 
        \hline
        32  & $1E-8$  & $3.39E-2$  &   $3.23E-4$  &   $9.63E-2$ & $1.86E-3$  &   $3.18E-3$  &   $2.79E-3$\\ 
        64  & $1E-8$  & $1.01E+0$  &   $9.45E-4$  &   $9.79E+0$ & $2.44E-2$  &   $1.20E-2$  &   $1.42E-2$\\ 
        128 & $1E-8$  & $5.04E+1$  &   $1.76E-3$  &   $1.00E+1$ & $3.78E-1$  &   $2.20E-2$  &   $1.05E-1$\\ 
        256 & $1E-8$  &      --    &   $8.65E-3$  &   $8.96E+1$ &      --    &   $1.47E-1$  &   $6.64E-1$\\ 
        512 & $1E-8$  &      --    &   $3.97E-2$  &   $7.84E+2$ &      --    &   $3.60E-1$  &   $3.74E+0$\\ 
        \hline
    \end{tabular}
    \caption{\label{tab:iso_con_time}Isotropic scattering case: computational cost of the three
      methods for different $n$ and $\epsilon$ with constant scattering coefficient $\mu_s=1$ and
      transport coefficient $\mu_t=1.2$.}
\end{table}


For a fixed mesh size $n$, when $\epsilon$ changes, the solution time of the FFT-based method
changes as well because $\epsilon$ affects the number of iterative steps of MINRES. Both the
precomputation time and the solution time of the RSF-based method are affected as well. When $\epsilon$
decreases, more skeletonization points are selected in the RSF, so both the precomputation and
solution times increase. It is worth to point out that even for $\epsilon=1E-8$, the computational
cost of RSF-based method is superior to the direct method if $n\geq64$. When $n\geq256$, due to the
limitation of RAM, the direct method fails to work any more.

Comparing these three methods, one observes that the FFT-based method is the fastest one for the
constant case. The solution time of RSF-based method is comparable with the FFT-based method and
its precomputation time is faster than the direct method if $n$ is not small.

\paragraph{Relative error}
\begin{table}[ht]
    \centering
    \begin{tabular}{rr|cc|cc|cc}
        \hline
        \rule{0pt}{3ex}
        $n$ & $\epsilon$ & $\cE^{\FFT}_{1}$ & $\cE^{\RSF}_{1}$ &
        $\cE^{\FFT}_{2}$ & $\cE^{\RSF}_{2}$ &
        $\cE^{\FFT}_{3}$ & $\cE^{\RSF}_{3}$ \\\hline
        32  & $1E-4$  & $2.66E-6$	&   $1.81E-5$	& $1.44E-5$	&   $1.85E-5$	& $2.10E-5$	&   $1.65E-5$	\\ 
        64  & $1E-4$  & $2.69E-6$	&   $2.67E-5$	& $1.44E-5$	&   $2.49E-5$	& $2.10E-5$	&   $2.23E-5$	\\ 
        128 & $1E-4$  & $2.70E-6$	&   $1.96E-5$	& $1.43E-5$	&   $1.88E-5$	& $7.69E-7$	&   $1.69E-5$	\\ 
        256 & $1E-4$  & $2.71E-6$	&   $1.89E-5$	& $3.81E-7$	&   $1.74E-5$	& $7.70E-7$	&   $1.49E-5$	\\ 
        512 & $1E-4$  & $2.72E-6$	&   $1.69E-5$	& $3.82E-7$	&   $1.49E-5$	& $7.71E-7$	&   $1.34E-5$	\\ 
        \hline
        32  & $1E-6$  & $8.96E-8$	&   $2.74E-7$	& $3.72E-7$	&   $2.77E-7$	& $7.49E-7$	&   $2.50E-7$	\\ 
        64  & $1E-6$  & $9.01E-8$	&   $3.06E-7$	& $1.64E-8$	&   $3.04E-7$	& $2.47E-8$	&   $2.68E-7$	\\ 
        128 & $1E-6$  & $9.04E-8$	&   $4.60E-7$	& $1.64E-8$	&   $4.64E-7$	& $2.49E-8$	&   $4.17E-7$	\\ 
        256 & $1E-6$  & $9.07E-8$	&   $4.37E-7$	& $1.64E-8$	&   $4.23E-7$	& $2.49E-8$	&   $3.73E-7$	\\ 
        512 & $1E-6$  & $3.12E-9$	&   $4.61E-7$	& $1.64E-8$	&   $4.49E-7$	& $2.49E-8$	&   $4.00E-7$	\\ 
        \hline
        32  & $1E-8$  & $3.02E-9$	&   $3.16E-9$	& $6.06E-10$	&   $3.01E-9$	& $1.03E-9$	&   $2.62E-9$	\\ 
        64  & $1E-8$  & $3.07E-9$	&   $1.25E-8$	& $7.01E-10$	&   $1.30E-8$	& $1.18E-9$	&   $1.19E-8$	\\ 
        128 & $1E-8$  & $5.09E-11$	&   $1.92E-8$	& $7.07E-10$	&   $1.94E-8$	& $1.20E-9$	&   $1.77E-8$	\\ 
        256 & $1E-8$  & $5.12E-11$	&   $1.73E-8$	& $7.07E-10$	&   $1.70E-8$	& $4.13E-11$	&   $1.53E-8$	\\ 
        512 & $1E-8$  & $5.14E-11$	&   $2.04E-8$	& $2.63E-11$	&   $2.04E-8$	& $4.13E-11$	&   $1.86E-8$	\\ 
        \hline
    \end{tabular}
    \caption{\label{tab:iso_con_err}Isotropic scattering case: relative error of the FFT-based
      method and RSF-based method for three source term in \eqref{eq:source} with constant
      scattering coefficient $\mu_s=1$ and transport coefficient $\mu_t=1.2$. $\cE_k$ is the
      relative error for the source term $f_k$, $k=1,2,3$.}
\end{table}
Simulations for different mesh size $n$ and different tolerance $\epsilon$ are performed to study
the relative error. Table \ref{tab:iso_con_err} presents the relative error of the FFT-based method
and the RSF-based method for the three source terms \eqref{eq:source} with $\mu_s=1$ and
$\mu_t=1.2$. The reference solution is the solution of the direct method if $n=32$, $64$ and $128$.
For $n>128$, the reference solution is computed with the FFT-based method with a very small relative
tolerance for MINRES. Both the two methods behavior well on the relative error.

\paragraph{Behavior for different scattering coefficients}
\begin{table}[ht]
    \centering
    \begin{tabular}{rr|cc|cc|cc}
        \hline \rule{0pt}{3ex}
        $\mu_s$ &   $\epsilon$ & $T_{\pre}^{\FFT}$ & $T_{\pre}^{\RSF}$ &
        $T_{\sol}^{\FFT}$ & $T_{\sol}^{\RSF}$ & $\cE^{\FFT}$ & $\cE^{\RSF}$ \\\hline
        1  &  $1E-4$  &  $1.78E-3$  &  $2.82E+0$  &  $1.51E-2$  &  $6.58E-2$  &  $2.70E-6$  &  $1.92E-5$\\	 
        5  &  $1E-4$  &  $1.78E-3$  &  $3.43E+0$  &  $2.16E-2$  &  $6.77E-2$  &  $1.60E-6$  &  $6.28E-5$\\	 
        10 &  $1E-4$  &  $1.79E-3$  &  $2.99E+0$  &  $3.41E-2$  &  $7.64E-2$  &  $1.90E-6$  &  $6.79E-4$\\	 
        \hline
        1  &  $1E-6$  &  $1.75E-3$  &  $4.87E+0$  &  $2.11E-2$  &  $7.66E-2$  &  $9.04E-8$  &  $5.28E-7$\\	 
        5  &  $1E-6$  &  $1.75E-3$  &  $4.96E+0$  &  $2.69E-2$  &  $8.08E-2$  &  $5.90E-8$  &  $5.42E-7$\\	 
        10 &  $1E-6$  &  $1.76E-3$  &  $4.94E+0$  &  $4.24E-2$  &  $8.47E-2$  &  $3.81E-8$  &  $2.81E-6$\\	 
        \hline
        1  &  $1E-8$  &  $1.75E-3$  &  $9.81E+0$  &  $2.19E-2$  &  $1.16E-1$  &  $5.09E-11$  &  $1.67E-8$\\	 
        5  &  $1E-8$  &  $1.76E-3$  &  $1.00E+1$  &  $3.24E-2$  &  $1.14E-1$  &  $3.46E-10$  &  $1.20E-8$\\	 
        10 &  $1E-8$  &  $1.75E-3$  &  $9.04E+0$  &  $4.57E-2$  &  $1.05E-1$  &  $9.82E-11$  &  $9.79E-9$\\	 
        \hline
    \end{tabular}
    \caption{\label{tab:iso_con_mus}Isotropic scattering case: computational cost and relative error
      for different scattering coefficient $\mu_s$ and $\epsilon$ with $n=128$ and
      $\mu_t=\mu_s+0.2$.}
\end{table}
Here we study the dependence on the scattering coefficient $\mu_s$ for the source term
\eqref{eq:source1}. When $\mu_s$ is small, for example $\mu_s=1$, the physics corresponds to the
transport regime for the photon. The case of relatively large $\mu_s$, for example $\mu_s=10$,
corresponds to the diffusive regime.  Here we always set $\mu_t = \mu_s + 0.2$.  The simulation
results for $\mu_s=1,5,10$ are listed in Table \ref{tab:iso_con_mus}. Because the use of the MINRES,
for different $\mu_s$, the solution time of the FFT-based method varies significantly as the
iterations are quite different. The RSF-based method is quite robust in terms of the scattering
coefficient. This indicates that both the two methods work in both diffusive regimes and transport
regimes and that the RSF-based method is quite insensitive to the parameters.

\subsubsection{Inhomogeneous media}
When the transport coefficient $\mu_t$ depends on $\bx$, the FFT-based method fails to work. Here we
test the RSF-based method for the variable transport coefficient case. The scattering coefficient
takes the form of
\begin{equation}\label{eq:mus}
  \mu_s(\bx)=1+\rho \exp\left( -\frac{|\bx-\bc|^2}{4} \right), \quad \bc = \left(\frac{1}{2},
  \frac{1}{2}\right),
\end{equation}
where $\rho$ is constant.

\paragraph{Computational cost}
\begin{table}[ht]
  \centering
  \begin{tabular}{rr|cc|cc}
    \hline \rule{0pt}{3ex}
    $n$ & $\epsilon$ & 
    $T_{\pre}^{\DIR}$(sec) & $T_{\pre}^{\RSF}$(sec) & 
    $T_{\sol}^{\DIR}$(sec) & $T_{\sol}^{\RSF}$(sec) \\
    \hline
    32  & $1E-4$ & $1.62E-1$ &  $1.41E-1$ & $2.51E-3$ &  $3.22E-3$\\ 
    64  & $1E-4$ & $1.94E+0$ &  $7.16E-1$ & $2.32E-2$ &  $1.33E-2$\\ 
    128 & $1E-4$ & $6.84E+1$ &  $5.01E+0$ & $4.13E-1$ &  $6.72E-2$\\ 
    256 & $1E-4$ &     --    &  $3.81E+1$ &     --    &  $3.56E-1$\\ 
    512 & $1E-4$ &     --    &  $2.81E+2$ &     --    &  $1.77E+0$\\ 
    \hline
    32  & $1E-6$ & $1.62E-1$ &  $1.67E-1$ & $2.51E-3$ &  $3.85E-3$\\ 
    64  & $1E-6$ & $1.94E+0$ &  $1.05E+0$ & $2.32E-2$ &  $1.40E-2$\\ 
    128 & $1E-6$ & $6.84E+1$ &  $8.66E+0$ & $4.13E-1$ &  $7.84E-2$\\ 
    256 & $1E-6$ &     --    &  $6.72E+1$ &     --    &  $4.35E-1$\\ 
    512 & $1E-6$ &     --    &  $5.20E+2$ &     --    &  $2.26E+0$\\ 
    \hline
    32  & $1E-8$ & $1.62E-1$ &  $1.84E-1$ & $2.51E-3$ &  $4.07E-3$\\ 
    64  & $1E-8$ & $1.94E+0$ &  $1.43E+0$ & $2.32E-2$ &  $1.46E-2$\\ 
    128 & $1E-8$ & $6.84E+1$ &  $1.35E+1$ & $4.13E-1$ &  $1.12E-1$\\ 
    256 & $1E-8$ &     --    &  $1.19E+2$ &     --    &  $6.15E-1$\\ 
    512 & $1E-8$ &     --    &  $9.59E+2$ &     --    &  $3.40E+0$\\ 
    \hline
  \end{tabular}
  \caption{\label{tab:iso_var_time}Isotropic scattering case: computational cost of the direct
    method and RSF-based method for different $n$ and $\epsilon$ with variable scattering
    coefficient \eqref{eq:mus} with $\rho=1$ and $\mu_t=\mu_s+0.2$.}
\end{table}
We perform the simulations for different mesh size $n$ and different tolerance $\epsilon$ for the
scattering term \eqref{eq:mus} with $\rho=1$. The computational costs of the direct method and the
RSF-based method are presented in Table \ref{tab:iso_var_time}. The source term is chosen to be
\eqref{eq:source1}. Since $\mu_t$ depends on $\bx$, one has to calculate a line integral in the
evaluation of each entry of the matrix $\bA$. However, because the factorization dominates the
computational cost for the two methods, the precomputation time is only a bit slower than the
constant scattering case.

Comparing these two methods, when $n\geq64$, the RSF-based method becomes faster than the direct
method on both the precomputation and solution parts.

\paragraph{Relative error}
\begin{table}[ht]
    \centering
    \begin{tabular}{rr|c|c|c}
        \hline \rule{0pt}{3ex}
        $n$ & $\epsilon$ & $\cE^{\RSF}_1$ & $\cE^{\RSF}_2$ & $\cE^{\RSF}_3$ \\
        \hline
        32  & $1E-4$ & $2.02E-5$ & $2.02E-5$ & $1.70E-5$\\ 
        64  & $1E-4$ & $4.19E-5$ & $3.03E-5$ & $2.31E-5$\\
        128 & $1E-4$ & $3.04E-5$ & $2.32E-5$ & $2.05E-5$\\
        \hline
        32  & $1E-6$ & $3.05E-7$ & $3.36E-7$ & $3.05E-7$\\
        64  & $1E-6$ & $2.88E-7$ & $3.11E-7$ & $2.35E-7$\\
        128 & $1E-6$ & $4.29E-7$ & $4.43E-7$ & $3.51E-7$\\
        \hline
        32  & $1E-8$ & $6.49E-9$ & $8.73E-9$ & $6.50E-9$\\
        64  & $1E-8$ & $6.43E-9$ & $6.71E-9$ & $6.01E-9$\\
        128 & $1E-8$ & $1.33E-8$ & $1.42E-8$ & $1.06E-8$\\
        \hline
    \end{tabular}
    \caption{\label{tab:iso_var_err}Isotropic scattering case: relative error of the RSF-based
      method for three source terms \eqref{eq:source} with variable scattering term \eqref{eq:mus}
      with $\rho=1$ and $\mu_t=\mu_s+0.2$. $\cE_k$ is the relative error for the source term $f_k$.}
\end{table}
Table \ref{tab:iso_var_err} presents the relative error for different mesh size $n$ and different
tolerance $\epsilon$ with the scattering term \eqref{eq:mus} with $\rho=1$ and
$\mu_t=\mu_s+0.2$. The reference solution is calculated by the direct method and the relative error
behaves similarly as the constant scattering case.

\paragraph{Behavior for different scattering coefficients}
\begin{table}[ht]
  \centering
  \begin{tabular}{rr|c|c|c}
      \hline \rule{0pt}{3ex}
    $\rho$ & $\epsilon$ & $T_{\pre}^{\RSF}$ 
    & $T_{\sol}^{\RSF}$ & $\cE^{\RSF}$\\\hline
    1  & $1E-4$ &  $4.70E+0$ & $6.68E-2$ &  $2.29E-5$\\	
    5  & $1E-4$ &  $5.57E+0$ & $6.84E-2$ &  $6.37E-5$\\	
    10 & $1E-4$ &  $4.92E+0$ & $6.95E-2$ &  $4.37E-4$\\	
    \hline
    1  & $1E-6$ &  $8.78E+0$ & $8.36E-2$ &  $4.11E-7$\\	
    5  & $1E-6$ &  $8.08E+0$ & $7.43E-2$ &  $6.02E-7$\\	
    10 & $1E-6$ &  $8.36E+0$ & $7.79E-2$ &  $2.53E-6$\\	
    \hline
    1  & $1E-8$ &  $1.46E+1$ & $1.08E-1$ &  $1.17E-8$\\	
    5  & $1E-8$ &  $1.45E+1$ & $1.11E-1$ &  $9.66E-9$\\	
    10 & $1E-8$ &  $1.38E+1$ & $1.05E-1$ &  $1.08E-8$\\	
    \hline
  \end{tabular}
  \caption{\label{tab:iso_var_mus} Isotropic scattering case: computational cost and relative
    error for different $\rho$ in \eqref{eq:mus} and $\epsilon$ with $n=128$ and
    $\mu_t=\mu_s+0.2$.}
\end{table}
Here we study the dependence of the computational cost and accuracy of the methods on the scattering
coefficient $\mu_s$ for the source term \eqref{eq:source1}. As before, $\mu_t = \mu_s + 0.2$. The
simulation results for $\rho=1,5,10$ in \eqref{eq:mus} are listed in Table
\ref{tab:iso_var_mus}. Similar as the constant transport coefficient case, the RSF-based method is
quite robust in terms of the scattering coefficient. This indicates that the RSF-based method works
well in both diffusive regimes and transport regimes and is insensitive to different regimes.



\section{Anisotropic scattering}\label{sec:alg2}

This section studies the RTE with anisotropic scattering. The start point is the observation that
the scattering term $\mu_s(\bx)\fint_{\bbS^{d-1}} \sigma(\bx,\bv\cdot\bv') \Phi(\bx,\bv')\dd\bv'$ is
a convolution on sphere. Thus the Fourier transformation for $d=2$ and the spherical harmonics
transformation for $d=3$ can be used to simplify the scattering term.


\subsection{Integral formulation}\label{sec:ani_integral}
The RTE \eqref{eq:RTE} has been well studied in \cite{case1963existence, LionsVol6,
  schlottbom2011forward, egger2014lp}. Here, we extend the derivation in Section
\ref{sec:iso_integral} to the anisotropic case. The extension is not trivial due to the complexity
of the scattering term.

For the anisotropic case, let us redefine the operator $\cA$ as
\begin{equation}
  (\cA\Psi)(\bx,\bv) := \fint_{\bbS^{d-1}}\sigma(\bx,\bv\cdot\bv')\Psi(\bx,\bv')\dd\bv',
\end{equation}
then the RTE \eqref{eq:RTE} can be solved as 
\begin{equation}\label{eq:ani_solution}
  \Phi(\bx,\bv) = \cT^{-1}\mu_s\cA\Phi+\cT^{-1}\cJ f,
\end{equation}
where $\cT^{-1}$ and $\cJ$ are same as those in Section \ref{sec:alg1}.

Notice $(\cA\Phi)(\bx,\bv) := \fint_{\bbS^{d-1}}\sigma(\bx,\bv\cdot\bv') \Phi(\bx,\bv')\dd\bv'$ is a
convolution on sphere. Since the formulas for $2D$ and $3D$ are quite different, we discuss them
separately.

\subsubsection{$2D$ case}

In 2D ($d=2$), $\bv=(\cos(\theta),\sin(\theta))$. The operator $\cA$ can be rewritten as
\begin{equation}
  (\cA\Phi)(\bx,\theta) = \fint_{0}^{2\pi}
  \sigma(\bx,\theta-\theta')\Phi(\bx,\theta')\dd\theta'.
\end{equation}
Applying Fourier transformation on $\Phi$ and $\sigma$ with respect to $\theta$ on $[0,2\pi]$ gives
rise to
\begin{equation}
  (\cA\Phi)(\bx,\theta) = 
  \sum_{k=-\infty}^{\infty}\hat{\sigma}(\bx,k)
  \hat{\Phi}(\bx,k)e^{\imag k\theta},
\end{equation}
where we take the following convention for the continuous Fourier transform
\begin{equation}\label{eq:ani_Fouriermode}
  \hat{\phi}(\bx,k)=(\cF\phi)(\bx,k):=\fint_{0}^{2\pi}\phi(\bx,\theta)e^{-\imag
    k\theta}\dd\theta, \quad \phi=\Phi, \sigma.
\end{equation}
Then the solution \eqref{eq:ani_solution} can be reformulated as
\begin{equation}\label{eq:ani_solution_2d}
  \cF\Phi = \cF\cT^{-1}\mu_s\cA\Phi +\cF\cT^{-1}\cJ f.
\end{equation}
When the scattering kernel $\sigma$ is smooth (as we have assumed), a few Fourier modes of $\sigma$ dominate, i.e.
\begin{equation}
  \sigma(\bx,\theta)\approx \sum_{k\in\bbM} \hat{\sigma}(\bx,k)e^{\imag
    k\theta},
\end{equation}
where $\bbM\subset\bbZ$ contains a small number of frequencies, typically centering at zero. Then
the operator $\cA$ can be approximated as
\begin{equation}\label{eq:ani_2d_scattering}
  (\cA\Phi)(\bx,\theta) \approx (\tilde{\cA}\Phi)(\bx,\theta):= 
  \sum_{k\in\bbM} e^{\imag k\theta} \hat{\sigma}(\bx,k)  \hat{\Phi}(\bx,k).
\end{equation}
Noticing that integral of $\Phi(\bx,\theta)$ with respect to $\theta$ is the mean local density, which is
usually important in application, we always assume $0\in\bbM$.
We introduce an operator as
\newcommand\cFinverse{\cF^{\dagger}}
\begin{equation}
  (\cFinverse\hat{\phi})(\bx,\theta)=\sum_{k\in\bbM}\hat{\phi}(\bx,k)e^{\imag k\theta},
\end{equation}
which satisfies $\cF \cFinverse\hat{\phi} = \hat{\phi}$, similar as the pseudo-inverse in linear
algebra.
Therefore, we have $\tilde{\cA}\Phi=\cFinverse(\hat{\sigma}\hat{\Phi})$.
Moreover, let $\hat{g}(\bx,k)=f(\bx)\delta_{k,0}$, then $\cJ f = \cFinverse\hat{g}$.

Replacing $\cA$ in \eqref{eq:ani_solution_2d} by $\tilde{\cA}$ results in a Fredholm integral
equation of the second kind with the following form
\begin{equation}\label{eq:ani_2d_Fredholm}
  \hat{\Phi}(\bx,k) =
  \cK\hat{\Phi}(\bx,k)+\tilde{\cK}\hat{g}(\bx,k),
\end{equation}
where the operators $\cK$ and $\tilde{\cK}$ are defined as
\begin{equation}
    \cK=\cF\cT^{-1}\cFinverse\mu_s\hat{\sigma},\quad \tilde{\cK}=\cF\cT^{-1}\cFinverse.
\end{equation}
Noticing that 
\begin{equation}\label{eq:ani_derivation}
    \begin{aligned}
      \tilde{\cK}\hat{\phi}(\bx,k) &= \sum_{k'\in\bbM}
      \fint_0^{2\pi} \int_0^{\tau(\bx,\bv)} 
        \exp\left(-\int_0^s \mu_t\left(\bx - s'\bv\right) ds'\right)
        \hat{\phi}(\bx-s\bv,k)
        e^{\imag (k'-k)\theta}\dd s\dd\theta\\
        &=\sum_{k'\in\bbM}\int_{\Omega}
        \tilde{K}(\bx,\by)
        e^{\imag(k'- k)\theta}\hat{\phi}(\by,k)\dd\by,
    \end{aligned}
\end{equation}
where $\theta = \arccos\left( \frac{\bx-\by}{|\bx-\by|} \right)$, and $\tilde{K}(\bx,\by)$ is the
kernel for the isotropic case defined in \eqref{eq:iso_operator2}, we have
\begin{equation}
  \tilde{\cK}\hat{\phi}(\bx,k) =
  \sum_{k'\in\bbM}\int_{\Omega}\tilde{K}(\bx,k,\by,k')\hat{\phi}(\by,k)\dd\by,\quad 
  \tilde{K}(\bx,k,\by,k') = \tilde{K}(\bx,\by)e^{\imag (k'-k)\theta},
\end{equation}
and
\begin{equation}\label{eq:ani_2d_operator}
    \cK\hat{\phi}(\bx,k) = 
    \sum_{k'\in\bbM}\int_{\Omega}K(\bx,k,\by,k')\hat{\phi}(\by,k)\dd\by,
    \quad
    K(\bx,k,\by,k') = K(\bx,\by)\hat{\sigma}(\by,k')e^{\imag(k-k')\theta},
\end{equation}
where ${K}(\bx,\by)$ is the kernel defined in \eqref{eq:iso_operator}. The kernel $\tilde{K}$ is
Hermitian while $K$ is not due to the existence of $\mu_s(\by)\hat{\sigma}(\by,k')$. If
$|\hat{\sigma}(\by,k')|$ is uniformly positive for
any $k'\in\bbM$, i.e. there exists a $c_0>0$ such that $|\hat{\sigma}(\by,k')|\geq c_0$ for any
$k'\in\bbM$, one can use the same technique as in \eqref{eq:iso_Fredholm2} to obtain a Hermitian
operator. However, the condition $\hat{\sigma}(\by,k')>0$ for any $k'\in\bbM$, $\by\in\Omega$ is
usually too strong to hold in general.

\subsubsection{$3D$ case}
For the case $d=3$, $\bv=(\sin(\theta)\cos(\psi),
\sin(\theta)\sin(\psi), \cos(\theta))$. Expressing the spherical
function $\Phi(\bx,\bv)$ and $\sigma(\bx,\theta)$ in terms of the
spherical harmonic basis, we obtain
\begin{equation}
    \begin{aligned}
        \Phi(\bx,\bv) &=
        \sum_{l=0}^{\infty}\sum_{m=-l}^l\hat{\Phi}(\bx,l,m)Y_l^m(\theta,\psi),\\
        \sigma(\bx,\theta) &=
        \sum_{l=0}^{\infty}\hat{\sigma}(\bx,l)Y_l^0(\theta),
    \end{aligned}
\end{equation}
where $Y_l^m(\theta,\psi)$ is the spherical harmonics, and 
\begin{equation}\label{eq:ani_3d_scattering}
    \hat{\phi}(\bx,l,m) =\cF\phi(\bx,l,m) 
    := \frac{1}{4\pi} \int_{0}^{2\pi}\int_0^{\pi}\phi(\bx,\theta,\psi)
    Y_l^m(\theta,\psi)\dd\theta\dd\psi,\quad 
    \phi=\Phi, \sigma.
\end{equation}
Then the operator $\cA$ can be expressed as
\begin{equation}
    (\cA\Phi)(\bx,\bv)=\sum_{l=0}^{\infty}\sum_{m=-l}^l\hat{\Phi}(\bx,l,m)\hat{\sigma}(\bx,l)
    \sqrt{\frac{4\pi}{2l+1}}Y_l^m(\theta,\psi).
\end{equation}
Similar to the $d=2$ case, if a few modes of $\sigma$ dominate, i.e.
\begin{equation}
  \sigma(\bx,\theta)\approx\sum_{l\in\bbM}\hat{\sigma}(\bx,l)Y_l^m(\theta,\psi),
\end{equation}
the operator $\cA$ can be approximated by
\begin{equation}
  (\cA\Phi)(\bx,\bv)\approx \tilde{\cA}\Phi(\bx,\bv) := 
  \sum_{l\in\bbM}\sum_{m=-l}^l\hat{\Phi}(\bx,l,m)\hat{\sigma}(\bx,l)
  \sqrt{\frac{4\pi}{2l+1}}Y_l^m(\theta,\psi).
\end{equation}
Similar to the derivation of the $d=2$ case, by replacing $\cA$ by $\tilde{\cA}$, one obtains a
Fredholm integral equation of the second kind with the form
\begin{equation}\label{eq:ani_3d_Fredholm}
    \hat{\Phi}(\bx,l,m) = \cK_{A}\hat{\Phi}(\bx,l,m)
    +\tilde{\cK} \hat{g}(\bx,l,m),
\end{equation}
where $\hat{g}(\bx,l,m)=f(\bx)\delta_{l,0}\delta_{m,0}$, and the operators $\cK$ and
$\tilde{\cK}$ are defined as
\begin{equation}
  \begin{aligned}
    \cK\hat{\phi}(\bx,l,m)&=
    \sum_{l'\in\bbM}\sum_{m'=-l'}^{l'}\int_{\Omega}K(\bx,l,m,\by,l',m')\hat{\phi}(\by,l',m')\dd\bx,\\
    K(\bx,l,m,\by,l',m')&=K(\bx,\by)\hat{\sigma}(\by,l') \sqrt{\frac{4\pi}{2l+1}}
    Y_l^m(\theta,\psi)Y_{l'}^{m'}(\theta,\psi),
  \end{aligned}
\end{equation}
and
\begin{equation}
  \begin{aligned}
    \tilde{\cK}\hat{\phi}(\bx,l,m)&
    =\sum_{l'\in\bbM}\sum_{m'=-l'}^{l'}\int_{\Omega}\tilde{K}(\bx,l,m,\by,l',m')\hat{\phi}(\by,l,m)\dd\bx,\\
    \tilde{K}(\bx,l,m,\by,l',m')&=\tilde{K}(\bx,\by)Y_l^m(\theta,\psi)Y_{l'}^{m'}(\theta,\psi).
  \end{aligned}
\end{equation}

We remark that if $\bbM=\{0\}$, then the Fredholm equations \eqref{eq:ani_2d_Fredholm} and
\eqref{eq:ani_3d_Fredholm} simplify to the ones in the isotropic scattering case.

\subsection{Existence theory}
In this subsection, we extend the existence theory for the isotropic case to the anisotropic
one. For simplicity, here we only study the case $d=2$, and all the conclusions in this subsection
hold for the $d=3$ case.

Define the vector function space $L^p(\Omega, \bbM)$ with $1\leq p\leq \infty$ as
\begin{equation}
  L^p(\Omega,\bbM) :=\left\{ \bv:\Omega\rightarrow\bbC^{\#\bbM}\;\text{such
    that}\;\bv(x)=(v_m(x))_{m\in\bbM}: v_m\in L^p(\Omega) \right\},
\end{equation}
with the norm given by 
\begin{equation}
  \normLpOmegaAni{p}{\bv} =   \left(  \int_\Omega  \norm{\bv(x)}_2^p dx   \right)^{1/p},
\end{equation}
and introduce the assumption:
\begin{enumerate}
\item[(C3)] $\sigma(\bx,\bv\cdot\bv')$ is non-negative and measurable and
  \begin{equation}
    \fint_{\bbS^{d-1}}\sigma(\bx,\bv\cdot\bv')\dd\bv'=1,
    \quad
    \text{for any }(\bx,\bv)\in\Omega\times\bbS^{d-1}.
  \end{equation}
\end{enumerate}
The main analytic result in the anisotropic case is the following theorem.
\begin{theorem}\label{thm:ani_existence}
  Under assumptions (C1), (C2) and (C3), for any $1\leq p\leq \infty$ and $f\in L^p(\Omega)$, the
  Fredholm equation \eqref{eq:ani_2d_Fredholm} admits a unique solution $\hat{\Phi}\in
  L^p(\Omega,\bbM)$ which satisfies
  \begin{equation}\label{eq:ani_unique}
    \normLpOmegaAni{p}{\hat{\Phi}}
    \leq \tau \exp\left( \tau\normLpOmega{\infty}{\mu_s} \right)
    \normLpOmega{p}{f}.
    \end{equation}
\end{theorem}
The proof line is similar to that of Theorem \ref{thm:existence}.  Before the proof, we first define
\begin{equation*}
  G(k,k';\bx) = \fint_0^{2\pi} \int_0^{\tau(\bx,\bv)} E^{[\mu_t]}(\bx,\bv,s)\mu_s(\bx) e^{\imag
    (k'-k) \theta}\dd s\dd\theta,
\end{equation*}
and the matrix function $\bG(\bx)=(G(k,k';\bx))$.  For any given $\bw\in\bbC^{M}$, a direct
calculation yields
\begin{equation*}
  \begin{aligned}
    \bw^{H}\bG\bw&=
    \fint_0^{2\pi}\int_0^{\tau(\bx,\bv)}E^{[\mu_t]}(\bx,\bv,s)
    \mu_s(\bx)\sum_{k\in\bbM}w_{k}^* e^{-\imag k\theta} 
    \sum_{k'\in\bbM}w_{k'}e^{\imag k'\theta}
    \dd s\dd\theta\\
    &\leq C \fint_0^{2\pi}\left|
    \sum_{k\in\bbM}w_ke^{\imag k\theta}\right|^2\dd\theta
    =C |\bw|^2,
  \end{aligned}
\end{equation*}
where the relation \eqref{eq:mu_inequality}, is used in the inequality. Therefore, we obtain
\begin{equation}\label{eq:ani_Gnorm}
  \norm{\bG}_2\leq C < 1.
\end{equation}

Similar to the isotropic case, we first study the case $p=\infty$ and obtain the following result.
\begin{lemma}\label{lem:Linfcontraction_ani}
  Under assumptions (C1), (C2) and (C3), the linear operator $\cK: L^\infty(\Omega,\bbM)\rightarrow
  L^\infty(\Omega,\bbM)$ is a contraction with
  \begin{equation}\label{eq:cKLinf_ani}
    \normOperatorLpAni{\infty}{\cK} \leq C < 1.
  \end{equation}
  Assume $f\in L^\infty(\Omega)$ and let $\hat{g}(\bx,k)=f(\bx)\delta_{k,0}$, then
    \begin{equation}\label{eq:KfLinf_ani}
        \normLpOmegaAni{\infty}{\tilde{\cK}\hat{g}}\leq \tau\normLpOmega{\infty}{f}.
    \end{equation}
\end{lemma}
\begin{proof}
  The condition (C3) indicates $|\hat{\sigma}(\bx,k)|\leq 1$,
  $k\in\bbM$.
  For any $\boldsymbol{\hat{\phi}}\in L^{\infty}(\Omega,\bbM)$, recalling
  \eqref{eq:ani_derivation} and $|\hat{\sigma}(\bx,k)|\leq 1$, we
  have 
  \begin{equation*}
    \begin{aligned}
      |(\cK\hat{\phi})(\bx,k)| &\leq
      \left| \fint_0^{2\pi}\int_0^{\tau(\bx,\bv)}E^{[\mu_t]}(\bx,\bv,s)
      \mu_s(\bx-s\bv)\sum_{k'\in\bbM}|\hat{\phi}(\bx-s\bv,k')|
      e^{\imag (k'-k)\theta}\dd s\dd\theta
      \right|\\
      &\leq\sum_{k'\in\bbM}G(k,k';\bx)
      \normLpOmega{\infty}{\phi(\bx,k')}\\
      &\leq C\normLpOmegaAni{\infty}{\phi}.
    \end{aligned}
  \end{equation*}
  Therefore, the operator $\cK$ is a contraction in $L^\infty(\Omega,\bbM)$ with
  $\normOperatorLpAni{p}{\cK}\leq C$.
  
  The definition of $\tilde{\cK}$ indicates
  \begin{equation*}
      |(\tilde{\cK}\hat{g})(\bx,k)| = \left|\int_{\Omega}\tilde{K}(\bx,\by)e^{-\imag k\theta}f(\by)\dd\by\right|
      \leq 
      \int_{\Omega}\tilde{K}(\bx,\by)|f(\by)|\dd\by
      =\tilde{\cK}|f| \leq\tau \normLpOmega{\infty}{f},
  \end{equation*}
  where \eqref{eq:KfLinf} is used in the last inequality.
\end{proof}

For the case $1\leq p<\infty$, notice that \eqref{eq:ani_2d_Fredholm} can be rewritten as $\hat{\Phi}
= \tilde{\cK}\left( \mu_s \hat{\sigma}\hat{\Phi}+\hat{g} \right)$ with $\hat{g}(\bx,k)=f(\bx)\delta_{k,0}$. 
By letting $\hat{\Psi}=\mu_s \hat{\sigma}\hat{\Phi}+\hat{g}$, the Fredholm equation \eqref{eq:ani_2d_Fredholm}
can be stated equivalently as 
\begin{equation}\label{eq:ani_Fredholm_w}
    \hat{\Psi}(\bx,k)=\mu_s\hat{\sigma}\tilde{\cK}\hat{\Psi}(\bx,k)+\hat{g}, \quad 
    \hat{\Phi}=\tilde{\cK}\hat{\Psi}.
\end{equation}
\begin{lemma}\label{lem:Lpcontranction_ani}
  Under assumptions (C1), (C2) and (C3), the linear operator $\mu_s\hat{\sigma}\tilde{\cK} :
  L^p(\Omega,\bbM)\rightarrow L^p(\Omega,\bbM)$, $1\leq p<\infty$ is a contraction with
  \begin{equation}\label{eq:cKLp_ani}
      \normOperatorLpAni{p}{\mu_s\hat{\sigma}\tilde{\cK}} \leq C,
  \end{equation}
  and 
  \begin{equation}\label{eq:uwLp_ani}
    \normLpOmegaAni{p}{\hat{\Phi}}\leq \tau\normLpOmegaAni{p}{\hat{\Psi}}.
  \end{equation}
\end{lemma}
\begin{proof}
  For any $\hat{\phi}\in L^p(\Omega,\bbM)$, since $|\hat{\sigma}(\bx,k)|\leq 1$, we have
  \begin{equation*}
    \begin{aligned}
        \normLpOmegaAni{p}{\mu_s\hat{\sigma}\tilde{\cK}\hat{\phi}} &\leq
        \normLpOmegaAni{p}{
            \mu_s(\bx)
            \sum_{k'\in\bbM}
            \fint_{0}^{2\pi}\int_0^{\tau(\bx,\bv)}
            E^{[\mu_t]}(\bx,\bv,s)e^{\imag
            (k'-k)\theta}\hat{\phi}(\bx-s\bv,k')\dd s\dd\bv}\\
            &=\normLpOmegaAni{p}{
                \sum_{k'\in\bbM}G(k,k';\bx)\hat{\phi}(\bx,k')}
            \leq C\normLpOmegaAni{p}{\hat{\phi}},
        \end{aligned}
    \end{equation*}
    where the relation \eqref{eq:ani_Gnorm} is used in the last inequality. Therefore,
    $\mu_s\hat{\sigma}\tilde{\cK}$ is a contraction in $L^p(\Omega,\bbM)$ and
    $\normOperatorLpAni{p}{\mu_s\hat{\sigma}\tilde{\cK}}\leq C$.

    The non-negativeness of $\mu_t$ indicates $0\leq E^{[\mu_t]}(\bx,\bv,s)\leq1$, which yields
    \begin{equation}
        \begin{aligned}
            \normLpOmegaAni{p}{\hat{\Phi}} 
            &=\normLpOmegaAni{p}{\sum_{k'\in\bbM}
            \fint_{0}^{2\pi}  \int_0^{\tau(\bx,\bv)} 
            E^{[\mu_t]}(\bx,\bv,s)
            e^{\imag(k'-k)\theta}\dd s\dd\bv\hat{\Psi}(\bx,k')}\\
            &\leq
            \normLpOmegaAni{p}{\sum_{k'\in\bbM}
            \fint_{0}^{2\pi}  \tau 
            e^{\imag(k'-k)\theta}\dd\bv\hat{\Psi}(\bx,k')}
            \leq \tau\normLpOmegaAni{p}{\hat{\Psi}}.
        \end{aligned}
    \end{equation}
\end{proof}
With these two lemmas, the Theorem \ref{thm:ani_existence} goes as follows.
\begin{proof}[Proof of Theorem \ref{thm:ani_existence}]
  This proof is similar to the proof of Theorem \ref{thm:existence}.  For the case $p=\infty$,
  noticing Lemma \ref{lem:Linfcontraction_ani}, and applying Banach's fixed point theorem
  \cite[Chapter 3, Theorem 3.2]{hunter2001applied}, we obtain that \eqref{eq:ani_2d_Fredholm} has a
  unique solution $\hat{\Phi}\in L^\infty(\Omega,\bbM)$ if $\cK\hat{g}\in L^\infty(\Omega,\bbM)$,
  which is guaranteed by \eqref{eq:KfLinf_ani}.
  
  For the case $1\leq p<\infty$, noticing Lemma \ref{lem:Lpcontranction_ani} and applying Banach's
  fixed point theorem, we obtain that \eqref{eq:ani_Fredholm_w} has a unique solution $\hat{\Psi}\in
  L^p(\Omega,\bbM)$, which indicates \eqref{eq:ani_2d_Fredholm} has a unique solution $\hat{\Phi}\in
  L^p(\Omega,\bbM)$.  The conclusion \eqref{eq:ani_unique} is a directly deduction of Lemma
  \ref{lem:Lpcontranction_ani}.  This completes the proof.
\end{proof}

By setting $p=2$ in Lemma \ref{lem:Lpcontranction_ani}, we can directly deduce the result.
\begin{corollary}\label{cor:spd_ani}
  Under assumptions (C1), (C2) and (C3), the linear operator $\cI-\cK:$
  $L^2(\Omega,\bbM)\rightarrow L^2(\Omega,\bbM)$ is positive definite.
\end{corollary}
\begin{proof}
    Since $\sigma$ is a real function, the operator $\mu_s\hat{\sigma}$ is a Hermitian operator.
    Noticing $\cK=\tilde{\cK}\mu_s\hat{\sigma}$, and $\tilde{\cK}$ is a Hermitian operator,
    for any $\phi\in L^2(\Omega,\bbM)$, we have
    \begin{equation*}
        \left\langle \phi, \cK\phi\right\rangle
        =\left\langle \tilde{\cK}\phi, \mu_s\hat{\sigma}\phi\right\rangle
        =\left\langle\mu_s\hat{\sigma}\tilde{\cK}\phi,\phi\right\rangle.
  \end{equation*}
  Using \eqref{eq:cKLp_ani} with $p=2$ and H{\"o}lder inequality, we can obtain 
  \begin{equation*}
    \left\langle \phi, \mu_s\tilde{\sigma}\tilde{\cK} \phi \right\rangle
    \leq C\normLpOmegaAni{2}{\phi}^2,
  \end{equation*}
  so 
  \begin{equation*}
      \left\langle 
      \phi, \left(\cI-\cK\right)\phi
      \right\rangle\leq (1-C)\normLpOmegaAni{2}{\phi}^2,
  \end{equation*}
  i.e. the operator $\cI-\cK$ is positive definite.
\end{proof}

\subsection{Fast algorithms}

In this subsection, we study the fast algorithms for the Fredholm equations
\eqref{eq:ani_2d_Fredholm} and \eqref{eq:ani_3d_Fredholm} and take the $2D$ case
\eqref{eq:ani_2d_Fredholm} as the example.  The framework of the algorithm can be split into 2
steps:
\begin{enumerate}
\item[T.1] Evaluate $(\tilde{\cK}f)(\bx,k)$;
\item[T.2] Given $\hat{r}(\bx,k)$, solve $\left( \cI-\cK
    \right)\hat{\Phi}(\bx,k)=\hat{r}(\bx,k)$, or equivalently $\hat{\Phi}(\bx,k) = \left(
  \cI-\cK \right)^{-1}\hat{r}(\bx,k)$.
\end{enumerate}

Following the discretization scheme used in the isotropic case, we introduce $u_{i,k}$ as an
approximation of $\hat{\Phi}(\bx_i,k)$, $\mu_{s,i}=\mu_s(\bx_i)$, and $f_i=f(\bx_i)$. The
discretization of anisotropic integral equations takes the form
\begin{equation}\label{eq:ani_discretization}
    u_{i,k} = \sum_{k'\in\bbM}\sum_{j=1}^NK_{i,k,j,k'} u_{j,k'}
    +\sum_{j=1}^NK_{i,k,j,0}\frac{f_j}{\mu_{s,j}}.
\end{equation}
For $K_{i,k,j,k'}$, where $i\not=j$, $K_{i,k,j,k'}$ is equal to the product of $K(\bx_i,k,\bx_j,k')$
and the cell volume. When $i=j$, since $K(\bx,k,\by,k')$ is singular at $\bx=\by$ the value
$K_{i,k,i,k'}$ is set to be integral of $K$ at the $i$-th cell. The treatments of the line integral
in $E(\bx,\by)$ and the integral of $K(\bx,\by)$ in a cell are the same as those of the isotropic case.

If one writes $u_{i,k}$ as a vector $\bu\in\bbC^{NM}$ and $K_{i,k,j,k'}$ into a matrix form
$\bK\in\bbC^{NM\times NM}$, then \eqref{eq:ani_discretization} can be written as
\begin{equation}\label{eq:ani_AU}
  \bA\bu=\bK\bdf,
\end{equation}
where $\bA=\bI-\bK$ and $\bdf$ is the vector form of $f_i\delta_{k,0}/\mu_{s,i}$.
Corollary \ref{cor:spd_ani} guarantees the matrix $\bA$ is positive definite.

Similar to the isotropic case, in order to solve \eqref{eq:ani_AU}, a direct method that factorizes
$\bA$ by LU or Cholesky decomposition takes $O((NM)^3)$ steps and in what follows we discuss two
fast algorithms for both homogeneous and inhomogeneous media.

\subsubsection{Homogeneous media and FFT-based algorithm}\label{sec:aniFFT}

For a homogeneous medium, the total transport coefficient $\mu_t$ and the scattering coefficient
$\mu_s$ are independent on $\bx$. As a result the kernel
\begin{equation}
    K(\bx,k,\by,k') = \frac{\mu_s\hat{\sigma}(k')}{|\bbS^{d-1}|}
    \frac{\exp\left(-\mu_t|\bx-\by|\right)}{|\bx-\by|}
    e^{\imag(k'-k)\theta}
    := \kappa(\bx-\by,k,k')
\end{equation}
with $\theta = \arccos\left( \frac{\bx-\by}{|\bx-\by|} \right)$ depends only on $\bx-\by$, $k$ and
$k'$. Hence the integral part
\begin{equation}
  C(\bx,k,k') := \int_{\Omega}\kappa(\bx-\by,k,k')\hat{g}(\by,k')\dd\by
\end{equation}
of $(\cK \hat{g})(\bx,k)$ is a convolution. Similar to the isotropic case, if the domain $\Omega$ is
rectangular, the FFT-based algorithm in Section \ref{sec:isoFFT} can be applied to the evaluation of
$C(\bx,k,k')$ for each $k,k'\in\bbM$.  Thus, in order to evaluate $(\cK\hat{g})(\bx,k)$, one just
needs to use the FFT-based algorithm in Section \ref{sec:isoFFT} to calculate $C(\bx,k,k')$ for each
$k,k'\in\bbM$ and then $(\cK\hat{g})(\bx,k)=\sum_{k'\in\bbM}G(\bx,k,k')$.  Once one has the algorithm
of applying $(\cK\hat{g})(\bx,k)$, Step T.2 can be solved by iterative method, for example
GMRES\cite{saad1986gmres}.

If one discretizes the domain uniformly by $n$ points in each direction, the computation cost to
evaluate $(\cK\hat{g})(\bx,k)$ is reduced to $O(M^2N\log(N))$ with $N=n^d$ and $M=\#\bbM$. Hence,
the computation costs of Steps T.1 and T.2 are $O(M^2N\log(N))$ and $(n_{iter}M^2N\log(N))$,
respectively, where $n_{iter}$ is the number of iteration steps in the iterative method.  So the
total computation cost is $O(n_{iter}M^2N\log(N))$. Before the evaluation of $(\cK\hat{g})(\bx,k)$,
one needs to calculate the Fourier modes of $\kappa(\bx,k,k')$ for each $k,k'\in\bbM$, so the
computation cost and storage memory of the precomputation are $O(M^2N\log(N))$ and $O(M^2N)$,
respectively.

We remark that the algorithm presented above can be directly applied on
\eqref{eq:ani_3d_Fredholm}. The only difference is $M$, which equals to $M=\sum_{l\in\bbM}(2l+1)$
for $3D$ case.

Another issue worth to remark is that the summation with repect to $k'$ in the evaluation of $\cK
\hat{g}$ can be also accelerated by FFT if $M$ is not small. The primary idea is that one can first
calculate $\hat{\sigma}\hat{g}$ and then the residue of the kernel $K(\bx,k,\by,k')$ only depends on
$\bx-\by$ and $k-k'$.

\subsubsection{Inhomogeneous media and RSF-based algorithm}\label{sec:aniRSF}

As is discussed in Section \ref{sec:isoRSF}, the FFT-based algorithm requires the total transport
coefficient $\mu_t$ be constant and the iterative method is used to solve Step T.2. For a
inhomogeneous medium, the RSF-based algorithm in Section \ref{sec:isoRSF} does not suffer from these
limitations. Here we apply RSF on the system \eqref{eq:ani_AU}.

One important issue is how to apply the RSF to the matrix $\bA$.  RSF needs to select the
skeletonization to approximate the factorization. A nature choice is for different Fourier modes at
the same position $\bx$ to use the same skeletonization. However, the different Fourier modes of
$\sigma$ may differ from others significantly. In this case, using the same skeletonization at each
position is not always the optimal choice. Here we regard {\em the pair of a position and a Fourier
  mode} as a {\em generalized point}, and select the skeletonization over such generalized
points. As a result, the Fourier modes selected at different points can vary.

As the size of the matrix is now $NM$, the precomputation (factorization) cost of RSF is $t_f$ in
\eqref{eq:time}, with $N$ replaced with $NM$. The computation cost of Steps T.1 and T.2 are both
$t_{a/s}$ in \eqref{eq:time} also with $N$ replaced with $NM$.

Finally, we remark that the RSF-based algorithm can be also directly applied on
\eqref{eq:ani_3d_Fredholm} without any difficulty.

\subsection{Numerical results}

This subsection performs numerical simulations for the anisotropic scattering case and reports the
results for different $n$, different tolerance $\epsilon$, and also different choices of the
coefficients $\mu_s$ and $\mu_t$.

\subsubsection{Homogeneous media}
Here the scattering kernel $\sigma$ is chosen to be
\begin{equation}\label{eq:sigma}
  \sigma(\theta) = 1 + \frac{1}{5}e^{\imag\theta} + \frac{1}{5}e^{-\imag\theta},
\end{equation}
and accordingly $\bbM = \{-1,0,1\}$.  The absorption coefficient $\mu_a=\mu_t-\mu_s$ is always set
as $\mu_a=0.2$.
\paragraph{Computational cost}
\begin{table}[ht]
    \centering
    \begin{tabular}{rr|ccc|ccc}
        \hline \rule{0pt}{3ex}
        $n$   & $\epsilon$            &
        $T_{\pre}^{\DIR}$(sec) & $T_{\pre}^{\FFT}$(sec)  &
        $T_{\pre}^{\RSF}$(sec)  & $T_{\sol}^{\DIR}$(sec)  &
        $T_{\sol}^{\FFT}$(sec) & $T_{\sol}^{\RSF}$(sec)  \\ \hline
        32  & $1E-4$ & $2.42E+0$  & $1.33E-3$ &  $1.84E+0$ & $4.16E-2$ &  $2.05E-2$ & $3.32E-2$\\ 
        64  & $1E-4$ & $6.15E+1$  & $1.23E-2$ &  $1.69E+1$ & $6.28E-1$ &  $6.54E-2$ & $6.87E-2$\\ 
        128 & $1E-4$ &     --     & $7.84E-2$ &  $1.21E+2$ &    --     &  $1.31E-1$ & $3.17E-1$\\ 
        256 & $1E-4$ &     --     & $3.35E-1$ &  $8.24E+2$ &    --     &  $6.87E-1$ & $1.53E+0$\\
        \hline
        32  & $1E-6$ & $2.42E+0$  & $1.33E-3$ &  $3.57E+0$ & $4.16E-2$ &  $2.46E-2$ & $1.98E-2$\\ 
        64  & $1E-6$ & $6.15E+1$  & $1.23E-2$ &  $3.48E+1$ & $6.28E-1$ &  $6.99E-2$ & $1.11E-1$\\ 
        128 & $1E-6$ &   --       & $7.84E-2$ &  $3.08E+2$ &   --      &  $1.47E-1$ & $5.77E-1$\\ 
        256 & $1E-6$ &   --       & $3.35E-1$ &  $2.71E+3$ &   --      &  $1.20E+0$ & $3.29E+0$\\
        \hline
        32  & $1E-8$ & $2.42E+0$  & $1.33E-3$ &  $3.83E+0$ & $4.16E-2$ &  $2.67E-2$ & $2.60E-2$\\ 
        64  & $1E-8$ & $6.15E+1$  & $1.23E-2$ &  $5.93E+1$ & $6.28E-1$ &  $8.45E-2$ & $1.59E-1$\\ 
        128 & $1E-8$ &   --       & $7.84E-2$ &  $5.80E+2$ &   --      &  $1.69E-1$ & $8.62E-1$\\ 
        256 & $1E-8$ &   --       & $3.35E-1$ &  $4.88E+3$ &   --      &  $1.67E+0$ & $6.02E+0$\\ 
        \hline                       
    \end{tabular}
    \caption{\label{tab:ani_con_time}Anisotropic scattering case: computational cost of the three
      methods for different $n$ and $\epsilon$ with constant scattering coefficient $\mu_s=1$ and
      transport coefficient $\mu_t=1.2$.}
\end{table}
Similar to the isotropic case, we perform simulation for different mesh size $n=32$, $64$, $128$ and
$256$ and different tolerance $\epsilon$ as $1E-4$, $1E-6$ and $1E-8$, for the constant scattering
coefficient $\mu_s=1$ and $\mu_t=1.2$. The results are presented in Table \ref{tab:ani_con_time}
with the source term set to \eqref{eq:source1}. All the conclusions for the isotropic case hold for
this case. 

\paragraph{Relative error}
\begin{table}[ht]
  \centering
  \begin{tabular}{rr|cc|cc|cc}
    \hline \rule{0pt}{3ex}
    $n$ & $\epsilon$ & $\cE^{\FFT}_{1}$ & $\cE^{\RSF}_{1}$ &
    $\cE^{\FFT}_{2}$ & $\cE^{\RSF}_{2}$ &
    $\cE^{\FFT}_{3}$ & $\cE^{\RSF}_{3}$ \\\hline
    32  & $1E-4$ & $8.15E-6$ &  $2.49E-5$	&$2.18E-5$ &  $2.55E-5$&	$3.11E-5$ &  $2.19E-5$\\ 	
    64  & $1E-4$ & $8.30E-6$ &  $8.28E-5$	&$2.26E-5$ &  $8.34E-5$&	$3.18E-5$ &  $7.67E-5$\\	
    128 & $1E-4$ & $2.94E-4$ &  $3.36E-4$	&$2.28E-5$ &  $3.34E-4$&	$3.20E-5$ &  $2.65E-4$\\	
    256 & $1E-4$ & $2.95E-4$ &  $1.34E-3$   &$5.92E-4$ &  $1.23E-3$&    $3.22E-5$ &  $9.64E-4$\\
    \hline
    32  & $1E-6$ & $3.75E-7$ &  $3.20E-7$	&$7.57E-7$ &  $3.11E-7$&	$6.11E-8$ &  $2.89E-7$\\	
    64  & $1E-6$ & $3.89E-7$ &  $1.06E-6$	&$7.77E-7$ &  $1.13E-6$&	$9.44E-7$ &  $9.54E-7$\\	
    128 & $1E-6$ & $3.93E-7$ &  $3.58E-6$	&$7.78E-7$ &  $3.72E-6$&	$9.51E-7$ &  $3.15E-6$\\	
    256 & $1E-6$ & $3.94E-7$ &  $1.20E-5$   &$7.78E-7$ &  $1.16E-5$&    $9.52E-7$ &  $9.07E-6$\\
    \hline
    32  & $1E-8$ & $3.20E-10$ &  $4.36E-9$	&$1.28E-9$ &  $4.52E-9$&	$1.88E-9$ &  $4.00E-9$\\	
    64  & $1E-8$ & $8.36E-9$ &  $2.18E-8$	&$1.42E-9$ &  $2.29E-8$&	$1.99E-9$ &  $1.99E-8$\\	
    128 & $1E-8$ & $8.41E-9$ &  $5.44E-8$	&$1.45E-9$ &  $5.69E-8$&	$2.01E-9$ &  $4.69E-8$\\	
	256 & $1E-8$ & $8.41e-9$ &  $1.52E-7$   &$5.19E-8$ &  $1.53E-7$&    $6.07E-8$ &  $1.32E-7$\\	
    \hline
  \end{tabular}
  \caption{\label{tab:ani_con_err}Anisotropic scattering case: relative error of the FFT-based
    method and RSF-based method for three source term in \eqref{eq:source} with constant scattering
    coefficient $\mu_s=1$ and transport coefficient $\mu_t=1.2$. $\cE_k$ is the relative error for
    the source term $f_k$, $k=1,2,3$.}
\end{table}
Table \ref{tab:ani_con_err} presents the relative error of the FFT-based method and the RSF-based
method for three source terms \eqref{eq:source} with $\mu_s=1$ and $\mu_t=1.2$. The reference
solution is obtained with the direct method if $n=32$ and $64$ and with the FFT-based method with a
sufficiently small relative tolerance for GMRES if $n>64$. Both the methods behaviors well in terms
of the relative error.

\paragraph{Behavior for different scattering coefficients}
\begin{table}[ht]
  \centering
  \begin{tabular}{rr|cc|cc|cc}
    \hline \rule{0pt}{3ex}
    $\mu_s$ &   $\epsilon$ & $T_{\pre}^{\FFT}$ & $T_{\pre}^{\RSF}$ &
    $T_{\sol}^{\FFT}$ & $T_{\sol}^{\RSF}$ & $\cE^{\FFT}$ & $\cE^{\RSF}$ \\\hline
    1  & $1E-4$ & $1.32E-2$&  $1.52E+1$ &$5.12E-2$ & $6.45E-2$ & $8.30E-6$&	$1.02E-4$\\ 
    5  & $1E-4$ & $1.42E-2$&  $1.64E+1$ &$9.68E-2$ & $6.87E-2$ & $6.69E-5$&	$5.40E-4$\\ 
    10 & $1E-4$ & $1.35E-2$&  $1.84E+1$ &$9.10E-2$ & $6.76E-2$ & $4.94E-5$&	$1.35E-3$\\ 
    \hline
    1  & $1E-6$ & $1.31E-2$&  $3.43E+1$ &$6.30E-2$ & $1.10E-1$ & $3.89E-7$&	$1.11E-6$\\ 
    5  & $1E-6$ & $1.50E-2$&  $3.58E+1$ &$1.24E-1$ & $1.18E-1$ & $7.52E-8$&	$2.34E-6$\\ 
    10 & $1E-6$ & $3.37E-2$&  $3.84E+1$ &$1.51E-1$ & $1.12E-1$ & $1.31E-7$&	$3.44E-6$\\ 
    \hline
    1  & $1E-8$ & $1.34E-2$&  $5.83E+1$ &$6.50E-2$ & $1.51E-1$ & $8.36E-9$&	$2.28E-8$\\ 
    5  & $1E-8$ & $1.38E-2$&  $5.97E+1$ &$1.31E-1$ & $1.60E-1$ & $5.12E-9$&	$2.31E-8$\\ 
    10 & $1E-8$ & $1.36E-2$&  $5.99E+1$ &$1.39E-1$ & $1.68E-1$ & $2.64E-9$&	$3.84E-8$\\ 
    \hline
  \end{tabular}
  \caption{\label{tab:ani_con_mus}Anisotropic scattering case: computational cost and relative error
    for different scattering coefficient $\mu_s$ and $\epsilon$ with $n=128$ and $\mu_t=\mu_s+0.2$.}
\end{table}
Table \ref{tab:ani_con_mus} presents the computational cost and relative error for different 
scattering coefficient $\mu_s$ and tolerance $\epsilon$ for the source term \eqref{eq:source1}.
One can get the same conclusion as the isotropic case that both the FFT-based method and RSF-based
method work in both the diffusive regime and the transport regime.

\subsubsection{Inhomogeneous media}
We test the RSF-based method for the variable transport coefficient case with the absorption
coefficient to be $\mu_a=0.2$ and the scattering term to be
\begin{equation}\label{eq:sigma_vai}
    \mu_s(\bx)\sigma(\bx,\theta) = 
    1 + \left( 2+e^{\imag\theta}+e^{-\imag\theta}
    \right)\rho\exp\left( -\frac{|\bx-\bc|^2}{4} \right),\quad 
    \bc=\left( \frac{1}{2},\frac{1}{2} \right),
\end{equation}
where $\rho$ is a constant.

Table \ref{tab:ani_var_time} and Table \ref{tab:ani_var_err} present the computational cost and the
relative error of the RSF-based method for different mesh size $n$ and different tolerance
$\epsilon$ for the scattering term \eqref{eq:sigma_vai} with $\rho=1$, respectively. Table
\ref{tab:ani_var_mus} presents the computational cost and relative error for different scattering
coefficient $\mu_s$ and tolerance $\epsilon$ for the source term \eqref{eq:source1} with $n=64$. One
can see that all the conclusions for the isotropic case hold for the anisotropic case.

\begin{table}[ht]
  \centering
  \begin{tabular}{rr|cc|cc}
        \hline \rule{0pt}{3ex}
    $n$ & $\epsilon$ & 
    $T_{\pre}^{\DIR}$(sec) & $T_{\pre}^{\RSF}$(sec) & 
    $T_{\sol}^{\DIR}$(sec) & $T_{\sol}^{\RSF}$(sec) \\
    \hline
    32  & $1E-4$ & $1.84E+0$ & $4.26E+0$ & $4.15E-2$ & $ 1.27E-2$\\	
    64  & $1E-4$ & $7.58E+1$ & $4.02E+1$ & $6.24E-1$ & $ 5.68E-2$\\	
    128 & $1E-4$ &     --    & $2.43E+2$ &   --      & $ 2.95E-1$\\	
    256 & $1E-4$ &     --    & $1.11E+3$ &   --      & $ 1.72E+0$\\	
    \hline
    32  & $1E-6$ & $1.84E+0$ & $8.34E+0$ & $4.15E-2$ & $ 1.89E-2$\\	
    64  & $1E-6$ & $7.58E+1$ & $1.31E+2$ & $6.24E-1$ & $ 8.12E-2$\\	
    128 & $1E-6$ &     --    & $7.64E+2$ &   --      & $ 4.55E-1$\\	
    256 & $1E-6$ &     --    & $3.03E+3$ &   --      & $ 2.96E+0$\\	
    \hline
    32  & $1E-8$ & $1.84E+0$ & $1.01E+1$ & $4.15E-2$ & $ 2.65E-2$\\	
    64  & $1E-8$ & $7.58E+1$ & $1.86E+2$ & $6.24E-1$ & $ 1.46E-1$\\	
    128 & $1E-8$ &     --    & $1.68E+3$ &   --      & $ 7.16E-1$\\	
    256 & $1E-8$ &     --    & $5.02E+3$ &   --      & $ 4.36E+0$\\	
    \hline
  \end{tabular}
  \caption{\label{tab:ani_var_time}Anisotropic scattering case: computational cost of the direct
    method and RSF-based method for different $n$ and $\epsilon$ with variable scattering
    coefficient \eqref{eq:sigma_vai} with $\rho=1$ and $\mu_t=\mu_s+0.2$.}
\end{table}

\begin{table}[ht]
  \centering
  \begin{tabular}{rr|c|c|c}
        \hline \rule{0pt}{3ex}
    $n$ & $\epsilon$ & $\cE^{\RSF}_1$ & $\cE^{\RSF}_2$ & $\cE^{\RSF}_3$ \\
    \hline
    32 & $1E-4$ & $4.45E-5$ & $4.49E-5$ & $4.01E-5$\\ 
    64 & $1E-4$ & $1.16E-3$ & $9.95E-4$ & $8.94E-4$\\
    \hline
    32 & $1E-6$ & $2.27E-7$ & $2.02E-7$ & $1.65E-7$\\
    64 & $1E-6$ & $6.11E-6$ & $5.00E-6$ & $3.53E-6$\\
    \hline
    32 & $1E-8$ & $1.00E-9$ & $9.41E-10$ & $7.88E-10$\\
    64 & $1E-8$ & $1.58E-8$ & $1.51E-8$ & $1.28E-8$\\
    \hline
  \end{tabular}
  \caption{\label{tab:ani_var_err}Anisotropic scattering case: relative error of the RSF-based
    method for three source terms \eqref{eq:source} with variable scattering term
    \eqref{eq:sigma_vai} with $\rho=1$ and $\mu_t=\mu_s+0.2$. $\cE_k$ is the relative error for
    the source term $f_k$.}
\end{table}


\begin{table}[ht]
  \centering
  \begin{tabular}{rr|c|c|c}
        \hline \rule{0pt}{3ex}
    $\rho$ & $\epsilon$ & $T_{\pre}^{\RSF}$ 
    & $T_{\sol}^{\RSF}$ & $\cE^{\RSF}$\\\hline
    1  & $1E-4$ & $9.88E+1$ & $1.19E-1$ & $1.16E-03$\\	
    5  & $1E-4$ & $6.73E+1$ & $5.69E-2$ & $1.25E-03$\\	
    10 & $1E-4$ & $4.58E+1$ & $5.64E-2$ & $3.08E-04$\\	
    \hline
    1  & $1E-6$ & $1.02E+2$ & $1.16E-1$ & $6.11E-06$\\	
    5  & $1E-6$ & $1.09E+2$ & $9.95E-2$ & $9.84E-06$\\	
    10 & $1E-6$ & $1.08E+2$ & $9.10E-2$ & $5.53E-06$\\	
    \hline
    1  & $1E-8$ & $1.79E+2$ & $1.48E-1$ & $1.58E-08$\\	
    5  & $1E-8$ & $1.89E+2$ & $1.50E-1$ & $1.82E-08$\\	
    10 & $1E-8$ & $1.90E+2$ & $1.54E-1$ & $2.28E-08$\\	
    \hline
  \end{tabular}
  \caption{\label{tab:ani_var_mus} Anisotropic scattering case: computational cost and relative
    error for different $\rho$ in \eqref{eq:sigma_vai} and $\epsilon$ with $n=64$ and
    $\mu_t=\mu_s+0.2$.}
\end{table}


\section{Conclusion}\label{sec:conclusion}

In this paper, we discussed the integral equation formulations and fast algorithms for the RTE with
homogeneous and inhomogeneous media and for isotropic and anisotropic scattering. The primary observation is
that a smooth scattering term can be approximated with a short convolution on unit sphere.
Numerical simulations show both the FFT-based method and the RSF-based method work well. For a
homogeneous medium, the FFT-based method is the method of choice due to the efficiency of the FFT. For
an inhomogeneous medium, the RSF-based method requires less solution time. Moreover, numerical simulations
show that the RSF-based method can be used in both transport and diffusive regimes.


To simplify the discussion, we have only studied the vacuum boundary condition and assumed that the
source term $f$ depends only on the spatial variable $\bx$. The proposed algorithm remains valid for
other boundary conditions and for anisotropic source terms $f(\bx,\bv)$. As the boundary condition
can be treated as a source term on the boundary, it is sufficient to just treat the anisotropic
source term. The only extra effort is the calculation $(\cF\cT^{-1}f)(\bx,k)$ and this can be
treated by solving
\begin{equation}
  \begin{aligned}
    \bv\cdot\nabla_{\bx}\Phi(\bx,\bv)+\mu_t(\bx)\Phi(\bx,\bv) &=f(\bx,\bv),\quad \text{ in }
    \Omega\times \bbS^{d-1},\\
    \Phi(\bx,\bv) &= 0,\quad \text{ on } \Gamma_{-},
  \end{aligned}
\end{equation}
which can be solved by existing methods, for example in \cite{gao2009fast,grella2011sparse}.

\section*{Acknowledgments}

The work of Y.F. and L.Y. is supported in part by the U.S. Department of Energy’s Advanced
Scientific Computing Research program under award DE-FC02-13ER26134/DE-SC0009409 and the National
Science Foundation under award DMS-1521830. The authors thank Kui Ren and Yimin Zhong for
constructive discussions.

\bibliographystyle{plain} \bibliography{rte.bib}
\end{document}